\documentclass[final,leqno]{siamltex}

\usepackage{amssymb}
\usepackage{amsmath}
\usepackage{graphicx}
\usepackage{tikz}
\usetikzlibrary{arrows}
\usepackage{caption}
\usepackage{graphicx}
\usepackage{caption}
\usepackage{float}
\usepackage[ruled]{algorithm2e}

\title{Algebraic Multilevel Methods for Markov Chains\thanks{This work was supported by Deutsche Forschungsgemeinschaft (DFG) through the Collaborative Research Center
CRC 1114 ``Scaling Cascades in Complex Systems'', Project (B03) ``Multilevel coarse graining of multiscale problems''. I would like to thank my supervisor Prof. Kornhuber for his helpful advice and academic encouragement.}}

\author{Lukas Polthier}
\begin{document}

\setlength\parindent{0pt}
\maketitle

\begin{abstract}
A new algebraic multilevel algorithm for computing the second eigenvector of a column-stochastic matrix is presented. The method is based on a deflation approach in a multilevel aggregation framework. In particular a square and stretch approach, first introduced by Treister and Yavneh, is applied. The method is shown to yield good convergence properties for typical example problems.
\end{abstract}



\pagestyle{myheadings}
\nocite{*} 
\markboth{ALGEBRAIC MULTILEVEL METHODS FOR MARKOV CHAINS}{LUKAS POLTHIER}

\section{Introduction}
The eigenvectors of Markov chains contain information about the important processes of a stochastic system. For the class of irreducible and aperiodic Markov chains, the eigenvector corresponding to the eigenvalue $1$ is unique and contains the limit information of the Markov chain called stationary distribution. Eigenvectors corresponding to eigenvalues close to $1$ encode the slow structural transitions, i.e. the essential information about the dynamics. Homogeneous, finite-state, discrete-time Markov chains are described by a column-stochastic transition matrix which can be analyzed by tools from linear algebra.
\vspace{1em}

Let $B \in \mathbb{R}^{n \times n}$ be a column-stochastic matrix, i.e. $1^TB = 1^T$, $B_{ij}\ge 0$.
The problem of finding the second eigenvector can be written algebraically as
\begin{align}
	Bx = \lambda_2 x, \label{Hauptgleichung}
\end{align}
where we assume that the eigenvalues of $B$ fulfill
\begin{align}
	1 = \lambda_1 > \lambda_2 > |\lambda_k| \label{AssumptionOnEigenvalues}
\end{align}
for all $k \ge 3$ and $\lambda_2 \in \mathbb{R}$. Note that $\lambda_2$ is not known a priori. In particular we can not simply solve a linear system of equations but must exploit the eigenvalue structure of $B$ to compute the second eigenvector.
\vspace{1em}

In relevant applications, Markov chains typically have a large number of variables which makes direct solvers to obtain the eigenvectors inapplicable. Instead one uses iterative solvers which tend to converge rather slowly. For many years, it is well known that algebraic multilevel methods can be used to drastically improve the convergence of iterative solvers \cite{BMR84}.
However the classical theory of algebraic multilevel methods is applicable only to a certain type of matrices, e.g. symmetric positive-definite matrices \cite{RS87}.
For Markov matrices, classical multilevel aggregation methods to compute the invariant measure converge rather slowly in numerical experiments.
In \cite{DMM10} and \cite{TY10} more sophisticated approaches that can drastically speed-up the convergence of the multilevel aggregation method for the first eigenvector are proposed.
Based on these ideas we present a new algorithm to compute the second eigenvector of a Markov matrix using a multilevel approach combined with a Wielandt deflation.

\section{Aggregation Multilevel for the Invariant Measure}
\label{chapterAggregationMultilevelMethods}
In this section we recall aggregation multilevel methods for Markov chains to compute the invariant measure. These methods have been used in the literature \cite{DMM08, HL94, TY10}.

The problem of finding the invariant measure, can be written as
\begin{align*}
	Bx = x
\end{align*}
with $x \neq 0$.
For $A:=I-B$ we may equivalently solve
\begin{align}
	Ax = 0 \label{EineGleichung}
\end{align}
with $|x| = 1$.
In fact we are looking for a non-trivial vector in the kernel of $A$. Here $A$ is an irreducible, singular M-matrix that has a strictly positive solution to (\ref{EineGleichung}) that is unique up to scaling. This property is important for the well-posedness of the multilevel framework.

\subsection{Relaxation Methods} \label{sectionRelaxationMethods}
Two common relaxation methods for (\ref{EineGleichung}) are Jacobi and power iteration.
Assume that $B$ has eigenvalues $\lambda_i$ with right eigenvectors $v^{(i)}$ that are ordered as in (\ref{AssumptionOnEigenvalues}).
In the following we analyze how the eigenvectors govern the efficiency of power and Jacobi iteration.
The former is given by the iteration
\begin{align}
	x^{(k+1)} = Bx^{(k)}.
\end{align}
Suppose the initial probability vector $x^{(0)}$ is spanned by the eigenvectors $v^{(i)} \in \mathbb{R}^n$ of $B$ for some coefficients $\alpha_i$.
\[
	x^{(0)} = \sum_{i=1}^n \alpha_i v^{(i)}
\]
Then the $k$-th iterate of power iteration can be expressed by
\begin{align} \label{eigenvaluesIterationMatrix}
	x^{(k)} = B^k x^{(0)}
			= \sum_{i=1}^n \alpha_i B^k v^{(i)}
			= \sum_{i=1}^n \alpha_i \lambda_i^k v^{(i)}
			= \left[ \alpha_1 v_1 + \sum_{i=2}^n \alpha_i \lambda_i^k v^{(i)} \right].
\end{align}
If the initial iterate $x^{(0)}$ was chosen such that $\alpha_1 \neq 0$, the power method will converge as $|\lambda_i| < 1$ for $i=2, \cdots, n$.
\vspace{1em}

Conversely, Jacobi iteration aims at solving the linear system $A\pi = 0$ rather than $B\pi = \pi$. We consider a splitting of the coefficient matrix
\[
	A = D - (L + U)
\]
into diagonal part $D$ and lower respectively upper triangular part $L, U$. In particular $D,L,U$ are non-negative matrices and we assume that $D$ is non-singular.\\
Then, up to rescaling, Jacobi iteration with damping parameter $\omega \in (0,1]$ is given by
\begin{align}
	x^{(k+1)} &= x^{(k)} - \omega D^{-1}Ax^{(k)},
\end{align}
with iteration matrix $G_{dJac} = I - \omega D^{-1}A$. It is the eigenvalues of $G_{dJac}$ that determine the rate of convergence of the method. In particular we can think of Jacobi iteration as a power iteration with iteration matrix $G_{dJac}$.\\

\subsection{Smoothing Property}
Particularly enlightening is the effect of the two iterative methods to the eigenvectors of the coefficient matrix. In equation (\ref{eigenvaluesIterationMatrix}) we see how the eigenvalues define to which extend the iterative method is effective in reducing a particular error.\\
This is also what we observe in the numerical experiments: In the left plot of figure \ref{figureSmoothErrors} we see that damped Jacobi iteration is very effective in reducing errors which correspond to eigenvalues of $B$ close to $-1$.
On the other hand, errors corresponding to eigenvalues of $B$ close to $1$ are hardly affected by the relaxation.
In the right plot of figure \ref{figureSmoothErrors} we see that power iteration is very effective in reducing errors corresponding to eigenvalues of $B$ with small absolute value. The eigenmodes corresponding to eigenvalues of $B$ with large absolute value are hardly affected by power iteration.\\
Altogether the iterative methods do not reduce all error modes equally effective. That is why we speak of \textit{smoothing} or \textit{relaxation} methods in the sense that they ``smooth'' out some errors while leaving others largely unchanged.

For symmetric, positive definite matrices, this effect called \textit{smoothing property} can be defined more precisely \cite[chapter 3]{Stu99}.

\begin{figure}
 \centering
 \includegraphics[scale=0.5, trim = 0 0 0 0]{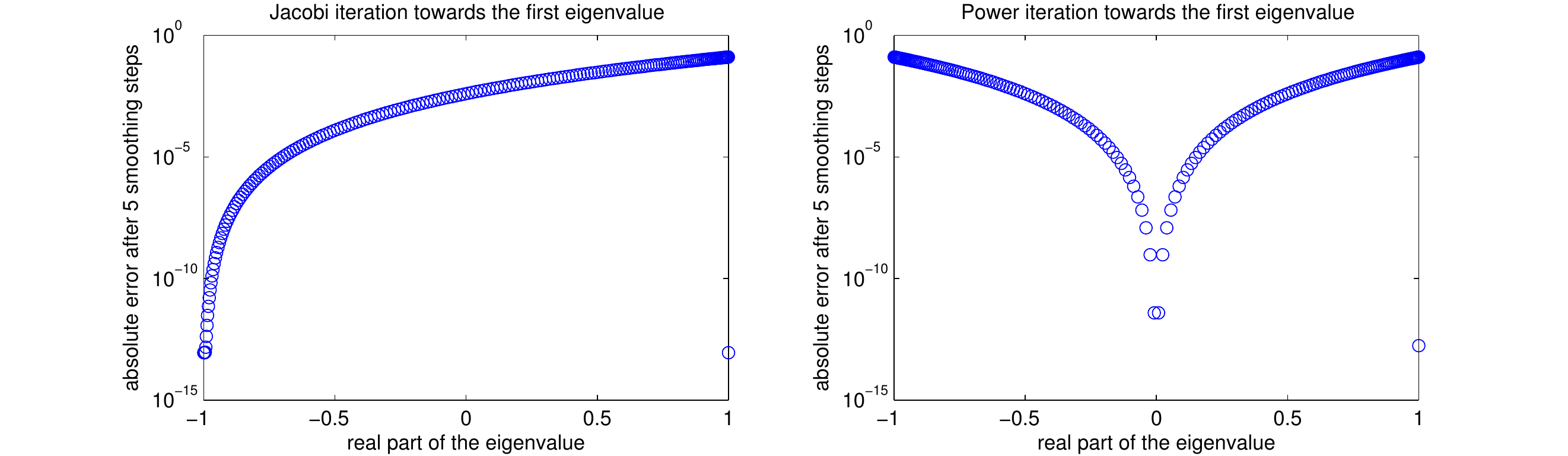}
 \caption{Illustration of the smoothing property of Jacobi and power iteration on the uniform chain matrix (section \ref{matrixUniformChain}). For each eigenvalue $\lambda_i$, $i=1, \cdots, n$, we take the initial iterate $x^{(0)} = v^{(1)} + \frac{1}{100} v^{(i)}$ consisting of the equilibrium distribution slightly perturbed in the direction of the $i$-th eigenvector. Then we plot the error $\|x^{(5)} - v^{(1)}\|_{l^1}$ to the equilibrium distribution after five relaxation steps.} \label{figureSmoothErrors}
\end{figure}

The basic idea of multilevel methods is to reduce the smooth error components on a smaller linear system.

\subsection{Aggregation Multilevel Methods}

As the exact solution $x$ to (\ref{EineGleichung}) has strictly positive components, it is reasonable to assume strict positivity of an approximation $x^{(i)}$ as well.
Then the multiplicative error of $x^{(i)}$ to $x$ is defined by 
\begin{align*}
	e^{(i)} := \left(\text{diag}(x^{(i)}) \right)^{-1} x.
\end{align*}
Then a multiplicative error formulation of (\ref{EineGleichung}) has the form
\begin{align}
	A \text{ diag}(x^{(i)})e^{(i)} = 0, \label{fineLevelProbEq}
\end{align}
where we have $e^{(i)} = 1 \in \mathbb{R}^n$ at convergence.
\vspace{1em}

From the theory on algebraic multilevel methods for positive definite matrices and the general experience with multilevel methods it turns out (see \cite{RS87}) that smooth errors $e$ have small variation along strong algebraic couplings. That means $e_i \approx e_j$ if $|A_{ij}|/A_{ii}$ is large.

Let $n_c < n$ and assume there is a partition of $\{1, \cdots, n\}$ into $n_c$ aggregates $\varphi_j$. These aggregates shall be chosen by some aggregation algorithm that aims at aggregating strongly connected elements.
Then, at least approximately, smooth errors are piecewise constant along the aggregates.
In \cite{TY10} a bottom-up aggregation approach is proposed, which seems to be a reasonable choice for the numerical example problems and will be used throughout this paper.
The choice of a particular aggregation method is typically motivated by heuristics and is not subject of this paper.
In fact the choice of aggregates may depend on the current iterate. 

Once reasonable aggregates have been determined, the aggregation matrix $Q\in \{0,1\}^{n\times n_c}$ defined by
\[
	Q_{ij} = \bigg \{ \begin{matrix}
	1 & \text{if } i \in \varphi_j\\
	0 & \text{if } i \not\in \varphi_j.
	\end{matrix}
\]
contains the information which fine level variables belong to which aggregate.

In particular $Q$ has full rank and up to reordering it has the following form.
\begin{align*}
	Q = \left[\begin{array}{c | c | c}
		1 & 0 & \cdots\\
		\vdots & \vdots & \cdots\\
		1 & 0 & \cdots\\
		\hline
		0 & 1 & \cdots\\
		\vdots & \vdots & \cdots\\ 
		0 & 1 & \cdots\\
		\hline
		\vdots & \vdots & \ddots
	\end{array}\right]
\end{align*}

On the coarser level we then solve for the error-components that are piecewise constant on each aggregate $\varphi_i$. If we replace $e^{(i)}$ by $Qe_c$ and multiply equation (\ref{fineLevelProbEq}) from the left by $Q^T$ we obtain the multiplicative error formulation on the coarse level
\begin{align}
	Q^TA ~\text{diag}(x^{(i)}) Qe_c = 0. \label{coarseLevelErrEq}
\end{align}
$e_c$ is a coarse level approximation of the (unknown) fine level multiplicative error $e^{(i)}$.
I.e. $e_c$ is the multiplicative error of the restriction $Q^Tx^{(i)}$ of the fine level iterate to the (unknown) exact coarse level solution $x_c$.
\begin{align*}
	x_c = \text{diag}(Q^Tx^{(i)}) e_c
\end{align*}

Together this gives the coarse level probability equation
\begin{align}
	Q^TA ~\text{diag}(x^{(i)}) Q \left(\text{diag}(Q^Tx^{(i)})\right)^{-1} x_c = 0 \label{coarseLevelProbEq}
\end{align}
with exact solution $x_c$.
\vspace{1em}

For simplicity of notation we define the prolongation and restriction operators by
\begin{align*}
	R &:= Q^T \in \mathbb{R}^{n_c \times n}\\
	P &:= \text{diag}(x^{(i)})Q \in \mathbb{R}^{n \times n_c}.
\end{align*}

In particular the current iterate $x^{(i)}$ is in the image of $P$.

In terms of grid-transfer operators $R, P$ we define the coarse level operator
\begin{align}
	A_c := RAP \label{coarseLevelOperatorDefinition}
\end{align}
and coarse level stochastic matrix 
\begin{align}
	B_c := RBP \left (\text{diag}(Rx^{(i)}) \right )^{-1}. \label{DefBc}
\end{align}

Now the coarse level multiplicative error equation (\ref{coarseLevelErrEq}) can be formulated by
\begin{align}
	A_c e_c = 0, \label{coarseLevelErrEq2}
\end{align}
and the coarse level probability equation (\ref{coarseLevelProbEq}) has the form
\begin{align}
	A_c \left(\text{diag}(Rx^{(i)}) \right)^{-1} x_c = 0. \label{coarseLevelProbEq2}
\end{align}

\begin{theorem}
\cite[theorem 3.1]{DMM10}
The coarse level matrix $A_c$ is again an irreducible, singular M-matrix. In particular (\ref{coarseLevelErrEq2}) has a unique solution.
\end{theorem}

We also get
\begin{align}
	\begin{split}
	A_c \left( \text{diag}(Rx^{(i)}) \right)^{-1}
	&= (R I P - R B P ) \left( \text{diag}(R x^{(i)}) \right)^{-1}
	\\
	&= I_c - B_c
	\end{split}
	\label{ABI}
\end{align}
where $I_c \in \mathbb{R}^{n_c\times n_c}$ denotes the identity matrix on the coarse level.

Hence by going from the fine level to the coarse level each aggregate $\varphi_j$ gets assigned the sum of its fine level weights, while going from coarse to fine level we weight the fine level variables of an aggregate according to the relative proportion of the approximated weights $x^{(i)}$.\\

By applying this approach in a recursive manner, we obtain the multilevel V-cycle described in algorithm AM (algebraic aggregation for Markov chains). 

\setlength{\algoheightrule}{0pt}
\setlength{\algotitleheightrule}{0pt}

\begin{algorithm}[H] 
  \SetAlgoLined
  \eIf{not on coarsest level}{
	$x \leftarrow Relax(A,x)$ (pre-relaxation: $\kappa$ smoothing-steps to reduce rough errors)\\
	Build aggregates $\varphi_j$ and aggregation matrix $Q$\\
	$R = Q^T$ and $P = \text{diag}(x)Q$\\
	$A_c = RAP$\\
	$x_c \leftarrow AM(A_c (\text{diag}(Rx))^{-1},Rx)$ (solve recursively on coarse level)\\
	$x \leftarrow P (\text{diag}(Rx))^{-1}x_c$ (correct fine level iterate by coarse level approximation)\\
	$x \leftarrow Relax(A,x)$ (post-relaxation: $\kappa$ smoothing-steps to reduce rough errors)
  }{
  	$x \leftarrow Solve(A,x)$ (solve $Ax = 0$ directly with an exact solver)
  }
  \caption{$AM(A,x)$}
\end{algorithm}

\begin{theorem}[Fixed point property] \label{thmFixedPointAggregation}
\cite[theorem 3.2]{DMM10}
The exact solution $x$ is a fixed point of the multilevel V-cycle, i.e. $AM(A,x) = x$.
\end{theorem}

\section{Deflated Square and Stretch Aggregation Multilevel for the Second Eigenvector}

In this chapter we introduce the new method to compute the second eigenvector of a stochastic matrix using a multilevel aggregation approach.
I.e. we want to solve (\ref{Hauptgleichung}) where we assume that the eigenvalues fulfill (\ref{AssumptionOnEigenvalues}).
Recall that $\lambda_2$ is not known a priori.

\subsection{Wieland Deflation}
\label{chapterWielandDeflation}
Consider the deflated matrix
\begin{align}
	B_1 := B - \mu v^{(1)} u^T, \label{deflatedMatrix}
\end{align}
where $u \in \mathbb{R}^n$ is some (for now arbitrary) deflation vector with $(v^{(1)},u)_{\mathbb{R}^n} = 1$ and $\mu \in \mathbb{R}$ is some (for now arbitrary) shift. Then the eigenvalues of $B_1$ are the same as those of $B$, except for the eigenvalue $\lambda_1$ that gets shifted by $\mu$.
 
\begin{lemma}[Wielandt deflation] \label{thmWielandtDeflation}
\cite[theorem 4.2]{Saa11}
Let $\{ v^{(1)}=\pi, v^{(2)}, \cdots , v^{(n)} \}$ denote the right eigenvectors corresponding to the eigenvalues $\{ \lambda_1, \lambda_2, \cdots, \lambda_n \}$ of a column-stochastic matrix $B$. 
Then the deflated matrix $B_1$ has the spectrum
\begin{align*}
	\sigma(B_1) = \{ \lambda_1 - \mu, \lambda_2 , \cdots ,\lambda_n \}.
\end{align*}
Moreover $B_1$ has the same left eigenvectors $u^{(i)}$ for $i=2,\cdots, n$ as $B$ and right eigenvectors of the form
\begin{align*}
	\tilde{v}^{(i)} = v^{(i)} - \gamma_i v^{(1)} %\label{c}
\end{align*}
for $i=2, \cdots, n$ with $\gamma_i = \gamma_i(u) := \frac{u^Tv^{(i)}}{1-(\lambda_1-\lambda_i)/\mu}$ depending on the deflation vector $u$.
\end{lemma}

Note that $v^{(1)} u^T \in \mathbb{R}^{n\times n}$ is in general a non-sparse matrix. But we can still employ matrix-vector multiplication by $B_1x = Bx - v^{(1)} (u^T x)$ using one sparse matrix-vector multiplication and one inner product evaluation.
For the special choice of a constant vector $u$, the right eigenvectors $v^{(2)}, \cdots, v^{(n)}$ are also preserved by the deflation. This particular choice $u \parallel 1$ for the deflation vector is also called Hotelling's deflation.
\vspace{1em}

Inspired by the deflation idea we can adapt the Jacobi iteration as follows.
First assume $\lambda >0$ such that the diagonal $D_1 = diag(E)$ of $E:=\lambda I -B_1$ is non-singular. Then Jacobi iteration with $E$ has the form
\begin{align}
	\begin{split}
	x^{(k+1)} &:= x^{(k)} - \omega D_1^{-1}(E x^{(k)})
	\\
	&~=
	x^{(k)} - \omega D_1^{-1} \left( \lambda x^{(k)} - B x^{(k)} + v^{(1)} ( u , x^{(k)} )_{\mathbb{R}^n} \right)
	\end{split}
	\label{JacobiIterationWielandt}
\end{align}
In a similar fashion we can combine power iteration with Hotelling's deflation:
\begin{align}
	x^{(k+1)} &:= B_1 x^{(k)} = B x^{(k)} - v^{(1)}(u^Tx^{(k)}) \label{PowerIterationWielandt}
\end{align}
For both methods we need to normalize after each iteration.

\begin{theorem}
For Hotelling's deflation with deflation vector $u \parallel 1$ and $\mu = \lambda_1$, the iteration (\ref{PowerIterationWielandt}) converges to the largest eigenvector of $B_1$ for any initial vector $x^{(0)} = \sum_{i=1}^n \alpha_i v^{(i)}$ with $\alpha_2 \neq 0$.
 The speed of convergence depends on $\frac{|\lambda_3|}{\lambda_2}$.
\end{theorem}
\begin{proof}
By lemma \ref{thmWielandtDeflation}, Hotelling's deflation preserves the right eigenvectors and eigenvalues $\lambda_2, \cdots, \lambda_n$. Then the convergence follows by a decomposition of the iterate into eigenvectors similar to the elaboration on classical power iteration in chapter \ref{sectionRelaxationMethods}.
\end{proof}

In particular, as Hotelling's deflation preserves the eigenvalues $\lambda_2, \cdots,\lambda_n$, the smoothing heuristic of deflated power iteration and deflated Jacobi iteration coincides with the heuristic for classical power and Jacobi iteration respectively.
Hence we can use the same multilevel heuristics as for the invariant measure.
\vspace{1em}

For the invariant measure, the corresponding eigenvalue was known a priori, namely $\lambda_1 = 1$. For the second eigenvector, the fine level problem has the form
\begin{align}
	B_1x = \lambda_2 x.
\end{align}
However $\lambda_2$ is not known a priori and we can not simply solve a linear system but instead we have to exploit that $\lambda_2$ is the largest eigenvalue of $B_1$.
That means we are seeking the eigenvector of $B_1$ corresponding to the eigenvalue with largest absolute value.

Also, the deflated matrix $B_1$ is not stochastic and even a full (i.e. non-sparse) matrix in general. We circumvent these difficulties by using the stochastic matrix $B$ for the multilevel setting and then solve for the second largest eigenvector of $B$ with deflated power iteration (\ref{PowerIterationWielandt}).
\vspace{1em}

The second eigenvector contains both positive and negative entries. This might cause a cancellation depending on the choice of aggregates: If positive and negative nodes in an aggregate sum to $0$, $\left(\text{diag}(Q^Tx^{(i)}) \right)^{-1}$ would not be well-defined.
Hence it is essential that we only aggregate elements that have the same sign in the approximation $x^{(i)}$ to the second eigenvector. Then the matrix 
\begin{align*}
	\text{diag}(x^{(i)}) Q \left(\text{diag}(Q^Tx^{(i)}) \right)^{-1}
\end{align*}
used to prolongate the coarse level approximation to the improved fine level iterate is non-negative.
In particular the coarse level probability matrix
\begin{align}
	B_c := Q^T B \text{ diag}(x^{(i)}) Q \left(\text{diag}(Q^Tx^{(i)}) \right)^{-1}. \label{coarseLevelMatrixSecondEV}
\end{align}
remains a non-negative stochastic matrix even though the prolongation operator
\begin{align*}
	P = \text{diag}(x^{(i)}) Q
\end{align*}
itself has both positive and negative entries.
From an intuitive viewpoint and for the exact solution, the coarse level problem should then preserve the second eigenvector as only nodes with the same sign in the second eigenvector are aggregated and weighted according to ratio in the second eigenvector. Hence the aggregation respects the slowest process of the Markov chain and does not distort it.
Our intuition is indeed consistent with the following theorem.

\begin{theorem} \label{thmcoarseLevelSecondEigenvectorPreserved}
Let $x$ be an eigenvector of the stochastic fine-level matrix $B$ with the eigenvalue $\lambda_2$. Moreover let the aggregates $\varphi_p$ be chosen such that $x_k x_l > 0$ for all $k,l \in \varphi_p$ for all aggregates $\varphi_p$. Then, for the exact solution $x^{(i)} = x$, the coarse level matrix $B_c$ defined in (\ref{coarseLevelMatrixSecondEV}) is a stochastic matrix with right eigenvector $Q^Tx$ corresponding to the eigenvalue $\lambda_2$.
\end{theorem}

\begin{proof}
By the	preceding remarks, the assumption on the aggregates ensures that $B_c$ is indeed a non-negative matrix. Moreover
\begin{align*}
	1_c^T B_c
	&= 1_c^T Q^T B \text{ diag}(x) Q \left(\text{diag}(Q^Tx) \right)^{-1}\\
	&= 1^T \text{ diag}(x) Q \left(\text{diag}(Q^Tx) \right)^{-1}\\
	&= x^T Q \left(\text{diag}(Q^Tx) \right)^{-1}\\
	&= (Qx)^T \left(\text{diag}(Q^Tx) \right)^{-1} = 1_c.
\end{align*}
Now simple calculations show that $Q^Tx$ is indeed an eigenvector.
\begin{align*}
	B_c (Q^Tx) &= Q^T B \text{ diag}(x) Q \left(\text{diag}(Q^Tx) \right)^{-1} (Q^Tx)\\
	&= Q^T B \text{ diag}(x) Q 1_c\\
	&= Q^T B \text{ diag}(x) 1 = Q^T B x = \lambda_2 (Q^Tx)
\end{align*}
\end{proof}

The resulting multilevel DAM (Wieland deflation with algebraic aggregation for Markov chains) V-cycle is described in algorithm \ref{AlgoSecondEigenvectorSimpleAggregation}.
First we compute the invariant  measure using a separate multilevel square and stretch algorithm from \cite{TY10} which is a modification of the algorithm described in section \ref{chapterAggregationMultilevelMethods}. Then the invariant measure is used in the relaxation method for a Hotelling's deflation for the modified power iteration (\ref{PowerIterationWielandt}). This step is denoted by $RelaxWieland(B,x,v)$.\\
The deflation happens only in the pre- and post-relaxation, while the coarse level probability matrix is still defined based on $B$. 
On the coarsest level, we apply deflated power iteration until the residual is close to machine precision.
This step is denoted by $Solve(B,x,v)$.

\begin{algorithm}[H] 
  \SetAlgoLined
  \eIf{not on coarsest level}{
  	$v \leftarrow S\&SM(B)$ (solve for the first eigenvector with the algorithm described in \cite{TY10})\\
	$x \leftarrow RelaxWielandt(B,x,v)$ (pre-relaxation)\\
	Build aggregation matrix $Q$ (aggregate only nodes with same sign)\\
	$R = Q^T$ and $P = \text{diag}(x)Q$\\
	$B_c = R B P (\text{diag}(R x) )^{-1} $\\
	$x_c \leftarrow DAM(B_c ,Rx)$ (solve recursively on coarse level)\\
	$x \leftarrow P (\text{diag}(R x) )^{-1} x_c$ (correct fine level iterate by coarse level approximation)\\
	$x \leftarrow RelaxWielandt(B,x,v)$ (post-relaxation)
  }{
  	$x \leftarrow Solve(B,x,v)$ (Solve $Bx = \lambda_2 x$ with plain relaxation)
  }
  \caption{$DAM(B,x)$}
  \label{AlgoSecondEigenvectorSimpleAggregation}
\end{algorithm}

Theorem \ref{thmcoarseLevelSecondEigenvectorPreserved} tells us that, for the exact solution, the second eigenvector $x$ on the fine level is preserved on the coarse level.
However for a fixed point property of the multilevel algorithm we additionally require that $(Q^Tx)$ is also the \textit{second} largest eigenvector on the coarse level.
This would be fulfilled if coarsening does not generate a new eigenvalue $\lambda_2 < \tilde{\lambda} < 1$ of $B_c$ that dominates $\lambda_2$.\\
From an intuitive viewpoint this is what we expect as aggregating some nodes should preserve the dynamics of the remaining nodes, in particular if the nodes are aggregated based on the strength of connection and the ratio in $x$.
\vspace{1em}

\textbf{Assumption on slow process.} If $x$ is the second eigenvector of $B$, assume that $(Rx)$ is the \textit{second} eigenvector of $B_c$.
\vspace{1em}

In the numerical results we observe that the assumption holds for all test problems and aggregates chosen by the bottom-up approach.
However, for arbitrary stochastic matrices and arbitrary aggregates, the assumption on slow processes does not hold in general. There are even counterexamples for small matrices ($n=5$) where the coarse level matrix has another eigenvalue that is larger than the second eigenvalue of the fine matrix.
It appears that this assumption is a property of the aggregation algorithm used in the multilevel framework.

\begin{theorem}
Assume the assumption on slow processes holds. Then the second eigenvector $x$ of $B$ is a fixed point of DAM.
\end{theorem}

\begin{proof}
This follows by the previous reasoning.
\end{proof}

For the multilevel V-cycle it remains to construct aggregates that satisfy the assumption $x_k x_l > 0$ for all $k,l\in \varphi_p$.
To this end we use the same bottom-up approach as for the classical method in \cite{TY10} where the strength of connection is defined by $B\text{ diag}( (x^{(k)})^+)$.
But when constructing the connectivity matrix $S$, we set those entries at $(k,l)$ to zero, where $x_k x_l \le 0$.
\vspace{1em}

\textbf{Some tricks to improve convergence.}
For the relaxation method, we additionally use Chebyshev iteration to influence the smoothing effect of the deflated power iteration. For a polynomial $p \in \mathbb{R}[x]$, Chebyshev iteration computes a vector of the form $x^{(k+1)} = p(B_1)x^{(k)}$. Here the polynomial $p$ should be chosen such that the desired eigenmodes are amplified, while others are damped.
Consider a splitting of the initial vector $x^{(k)} = \sum_{i=1}^n \alpha_i v^{(i)}$ into eigenvectors of $B_1$. Then it holds
\[
	p(B_1)x^{(k)} = \sum_{i=1}^n \alpha_i p(B_1)v^{(i)} = \sum_{i=1}^n \alpha_i p(\lambda_i) v^{(i)}. 
\]
To this end, we seek a polynomial which is large at $\lambda_2$ and small on the set $\{ \lambda_3, \cdots, \lambda_n \}$. Without the knowledge of the eigenvalues such a polynomial is not possible to compute. In practice however we can use e.g. third order polynomials with two zeros in $[-1,0]$ that hence strongly damp the rough eigenmodes.
Then coarse level correction reduces the smooth eigenmodes more effectively.
\vspace{1em}

To speed up the convergence in the first V-cycle iteration (only in the first!) it turns out to be more efficient if we use the matrix $B\text{ diag}(v^{(1)})$ to define the strength of connection and the prolongation operator
\[
	\text{diag}(v^{(1)})Q \left(\text{diag}(Q^Tv^{(1)}) \right)^{-1}
\]
to compute the fine level approximation from the coarse level iterate.
The reason behind this is that in the first iteration the (known) first eigenvector provides a good classification of the smooth errors, while the initial approximation $x^{(1)}$ to the second eigenvector still has a large error.\\
Once the iterate $x^{(k)}$ becomes sufficiently accurate, the strength of connection and the definition of the prolongation operator is based on the iterate $x^{(k)}$ as described in algorithm DAM.\\

In figure \ref{figureWielandtMultilevelVisualization} we see how the multilevel approach speeds up the convergence of plain vector iteration. In the first V-cycle there is a drastic improvement of the residual due to the coarse grid correction. As explained in the previous paragraph, this is due to the knowledge of the first eigenvector up to machine precision.\\
In the subsequent iterations, the coarse level correction introduces new errors that are then effectively reduced by the first post-relaxations.
This implies the coarse level correction improves the smooth error components, but introduces new rough error modes which then require damping by the post-relaxation steps.
All in all the residual is significantly reduced due to the coarse level correction.
For larger matrices however, the method becomes less efficient as the convergence rated increases for larger $n$.
This was also a property of the classical AM-algorithm for the first eigenvector.

\begin{figure}
\centering
 \includegraphics[scale=0.65, trim = 0 0 0 0]{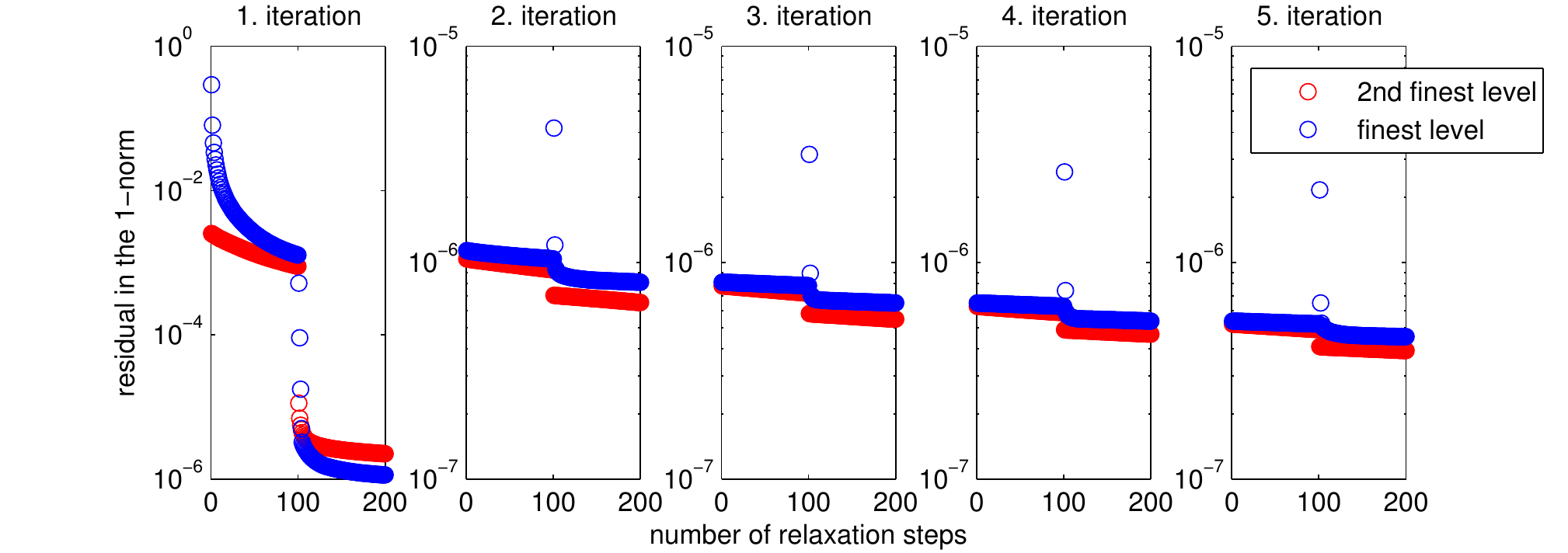}
 \caption{Illustration of multilevel improvement of DAM on the uniform chain matrix for $n=1024$ with 100 pre- and post-relaxation steps. The blue dots denote the residual in the one norm on the fine level after each relaxation step. After 100 pre-relaxation steps the coarse level correction is performed, then there are 100 post-relaxation steps. The red dots denote the residual on the second finest level in the V-cycle. (In total there are 6 levels in the V-cycle.)}
\label{figureWielandtMultilevelVisualization}
\end{figure}

\section{Square and Stretch}
In \cite{TY10, DMM10} it was observed that the simple grid-transfer operators lead to slow convergence of AM for the invariant measure. This was explained by the observation that rough errors are generated by $P$ and not damped properly by $R$. In \cite{DMM10} De Sterck et al propose smoothing the prolongation and restriction operator to improve the convergence. A square and stretch approach where we can still use the same grid-transfer operators but reduce the rough eigenmodes on the coarse level has been proposed in \cite{TY10} by Treister and Yavneh. In this section we will apply the square and stretch approach in a similar fashion to improve the convergence of DAM.\\

Regarding the first eigenvector, a matrix of the form $\hat{B} = \frac{1}{1-d}B - \frac{d}{1-d}I$ was considered where $d \in (0,1)$ is a stretching parameter. This changes the spectrum of the matrix in an advantageous way (see \cite{TY10}).\\

Here however, the eigenvalue $\lambda_2$ of interest is smaller than $1$. Thus if there is an eigenvalue $\lambda_k$ close to $0$ we may observe
\[
	\left| \frac{\lambda_k^2}{(1-d)} - \frac{d}{(1-d)} \right| > \left| \frac{\lambda_2^2}{(1-d)} - \frac{d}{(1-d)} \right|.
\]
Hence after the square and stretch transformation, there would be a different eigenvector whose eigenvalue has largest absolute value. This would distort the method.

There are two remedies: The first approach is to use pre- and post-relaxations such that the eigenvalues close to $0$ are eliminated up to machine precision. Then they do not distort the convergence.\\
The second approach is to use a shifting parameter $p \ge 1-\lambda_2$ that shifts the spectrum before squaring. The spectrum of the resulting matrix
\begin{align*}
	\hat{B} := \frac{1}{1-d} \left(\frac{1}{(1+p)^2}(B + pI)^2\right) - \frac{d}{1-d}I
\end{align*}
is still contained in the interval $[-1,1]$ but now $\lambda_2$ remains the dominant second eigenvalue in terms of largest absolute value. Of course the shifting parameter $p$ is not known in advance. Numerically we may use the approximation of the iterate to $1-\lambda_2$ with some multiplicative error correction.

Algorithm DS\&SM (Wieland deflation with square and stretch for Markov chains) describes the resulting multilevel V-cycle. In figure \ref{figureWielandtMultilevelVisualizationSquareStretch} we see that the square and stretch transformation drastically improves the convergence compared to the simple aggregation procedure depicted in figure \ref{figureWielandtMultilevelVisualization}.
On the other hand we observe that the coarse level correction now significantly introduces new errors.
Hence we require more relaxation steps to damp the new errors. The steep descent of the residual in the first post-relaxation steps indicates that coarse level correction is quite effective in reducing the smooth error components and that the new errors generated by the coarse level are only rough modes.

\begin{algorithm}[H] 
  \SetAlgoLined
  \eIf{not on coarsest level}{
  	$v \leftarrow S\&SM(B)$ (solve for the first eigenvector with the algorithm described in \cite{TY10})\\
	$x \leftarrow RelaxWielandt(B,x,v)$ (pre-relaxation)\\
	Build aggregates $\varphi_j$ and aggregation matrix $Q$\\
	$R = Q^T$ and $P = \text{diag}(x)Q$\\
	Determine shifting parameter $p$\\
	Determine stretching parameter $d$\\
	$\hat{B}_c = R \left( \frac{1}{1-d} \frac{1}{(1+p)^2} (B + pI)^2 - \frac{d}{1-d} I \right) P (\text{diag}(R x) )^{-1} $\\
	$x_c \leftarrow DS\&SM(\hat{B}_c ,Rx)$ (solve recursively on coarse level)\\
	$x \leftarrow P (\text{diag}(Rx))^{-1} x_c$ (correct fine level iterate by coarse level approximation)\\
	$x \leftarrow RelaxWielandt(B,x,v)$ (post-relaxation)
  }{
  	$x \leftarrow Solve(B,x,v)$ (Solve $Bx = \lambda_2 x$ directly)
  }
  \caption{$DS\&SM(B,x)$}
  \label{AlgoSecondEigenvectorSquareStretch}
\end{algorithm}

This algorithm has the same fixed point property as DAM.

\begin{figure}
\centering
 \includegraphics[scale=0.65, trim = 0 0 0 0]{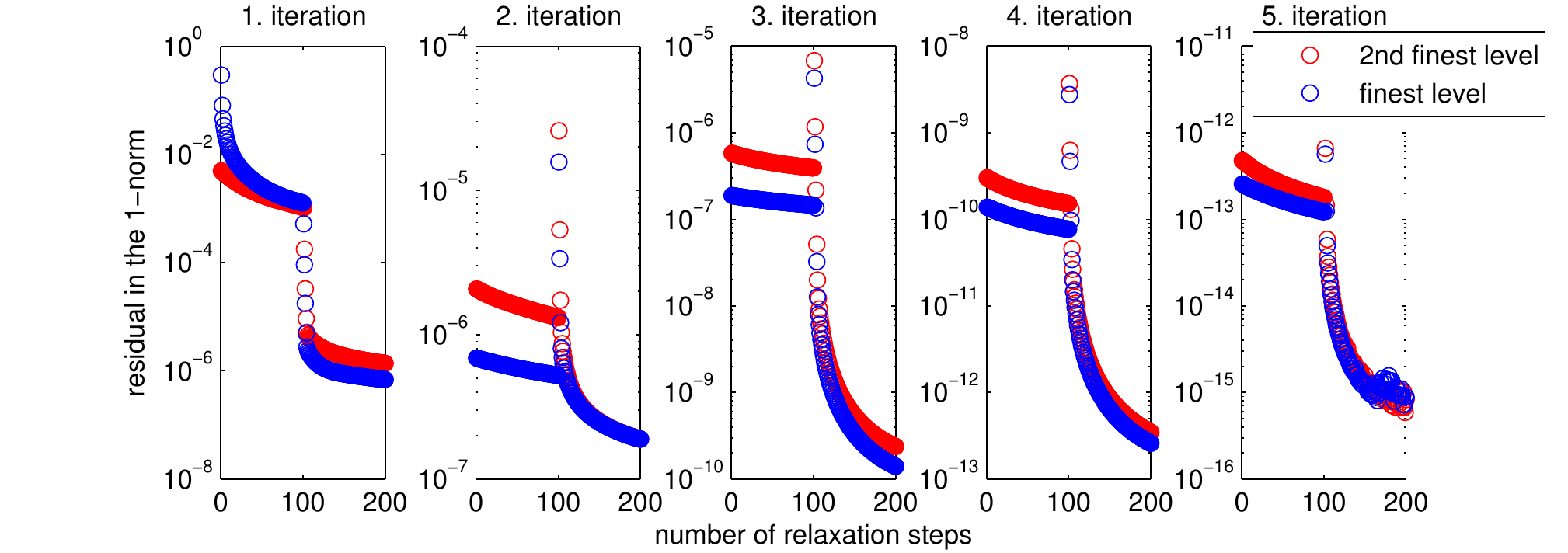}
 \caption{Illustration of the multilevel improvement of DS\&SM on the uniform chain matrix for $n=1024$ with 100 pre- and post-relaxation steps. The blue dots denote the residual in the one norm on the fine level after each relaxation step. After 100 pre-relaxations the coarse level correction is performed, then there are 100 post-relaxation steps. The red dots denote the residual on the second finest level in the V-cycle.}
\label{figureWielandtMultilevelVisualizationSquareStretch}
\end{figure}

\section{Numerical Results}
The choice of test problems was motivated by the examples used in \cite{DMM10, TY10}. Also the numerical parameters are chosen alike. $\gamma$ denotes the convergence rate of a V-cycle. $it$ denotes the number of steps to reduce the residual by factor of $10^{-10}$. $C_{op}$ denotes the operator complexity, i.e. the number of nonzero elements on all levels divided by the number of non-zero elements on the fine level.
$\gamma_{\text{eff}} = \left((10^{-10})^{1/it} \right)^{1/C_{op}}$ denotes the effective convergence factor to measure the overall effectiveness of the method. $lev$ denotes the number of levels of a V-cycle.
The parameter $\theta$ is used in the bottom-up aggregation scheme described in \cite{TY10}.

\subsection{Uniform Chain}
\label{matrixUniformChain}

The simplest test matrix we consider is the uniform chain matrix. Figure \ref{figUniformChain} illustrates the graph of the 1D chain with uniform weights. The connectivity matrix is symmetric with non-zero entries just above and below the main diagonal. The resulting transition matrix has eigenvalues that are almost uniformly distributed in the interval $[0,1]$.\\
Aggregates of size $s=2$ with aggregation parameter $\theta = 0.1$ are chosen.
For DS\&SM using $d=\min(\text{diag}(A))$ or $d=0.5$ gives the same results, as the diagonal elements are approximately $0.5$.
DS\&SM and DAM require a larger number of smoothing steps compared to AM in order to damp the rough error modes generated by the coarse level correction. We use 100 pre- and post-relaxations where each relaxation consists of 3 matrix-vector multiplications.\\
In table \ref{tableUniformChain2} we see that DS\&SM performs extremely well with bounded operator complexity and low convergence factors. Just as for the invariant measure, the square and stretch transformation significantly increases the performance.\\
In fact, despite the high number of relaxation steps, the good convergence factor of DS\&SM is the result of coarse level improvement together with relaxations.
If we consider only relaxation steps without the multilevel correction, the residual is merely reduced with a convergence factor of $0.999$ for 200 relaxation steps on a 16384x16384 uniform chain matrix.

\begin{figure}[]
\scalebox{0.6}{
\begin{tikzpicture}[->,>=stealth',shorten >=1pt,auto,node distance=1.8cm,
                    thick,main node/.style={circle,draw,font=\sffamily\Large\bfseries}]
  \coordinate (cone) at (0,0);
  \node[main node, minimum size = 0cm, draw=none] (0) [] {}; %fix place origin
  \node[main node, minimum size = 0.5cm] (1) at (0,5) {};
  \node[main node, minimum size = 0.5cm] (2) [right of=1] {};
  \node[main node, minimum size = 0.5cm] (3) [right of=2] {};
  \node[main node, minimum size = 0.5cm] (4) [right of=3] {};
  \node[main node, minimum size = 0.5cm] (5) [right of=4] {};

  \path[every node/.style={font=\sffamily\small}]
    (1) edge [left] node[above] {} (2)
    (2) edge [left] node[above] {} (3)
        edge [left] node[below] {} (1)
    (3) edge [left] node[above] {} (4)
        edge [left] node[below] {} (2)
    (4) edge [left] node[above] {} (5)
        edge [left] node[below] {} (3)
   	(5) edge [left] node[below] {} (4);
\end{tikzpicture}}
\includegraphics[scale=0.4]{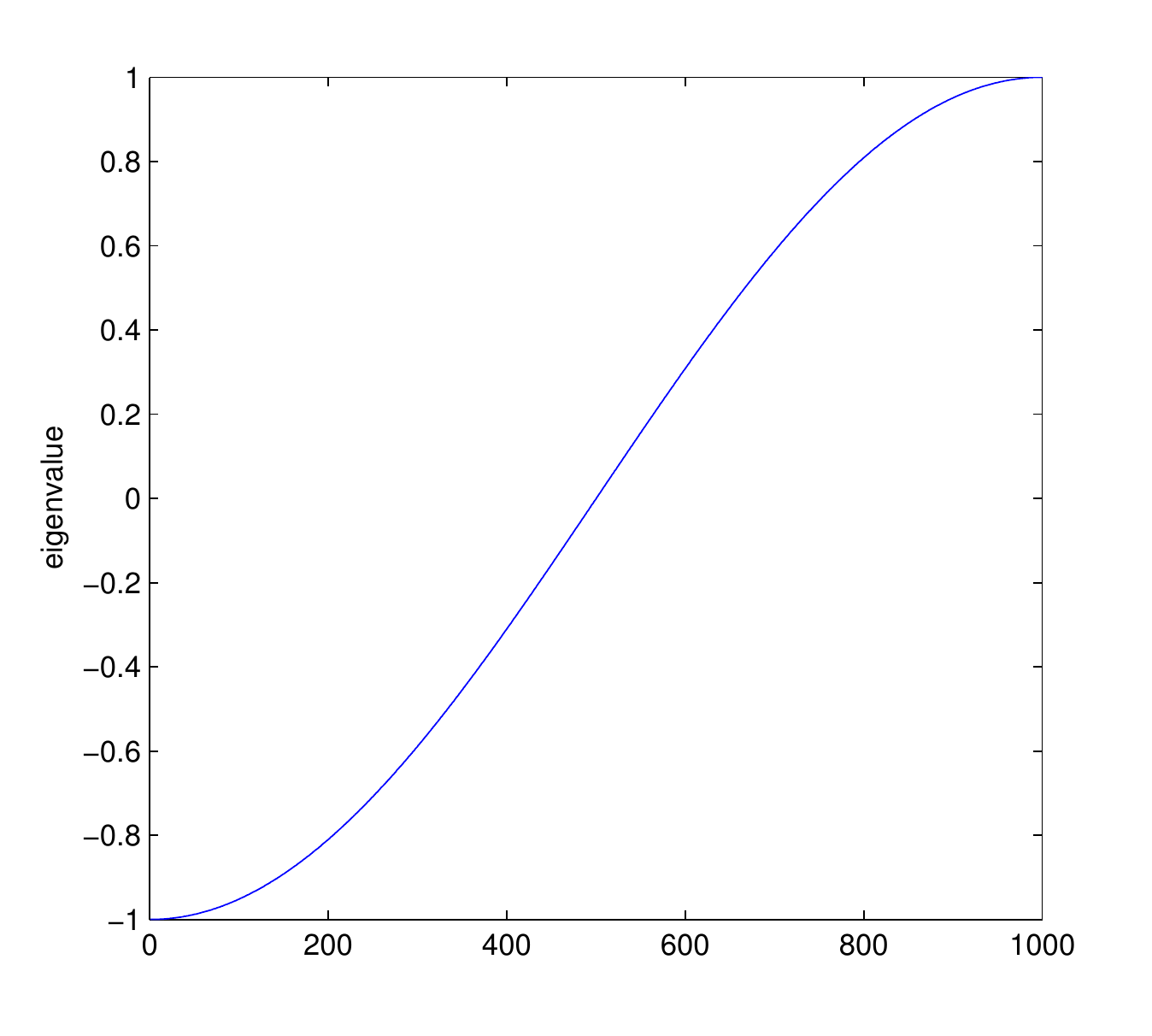}
\caption{Left: Graph representation of the uniform chain. Right: Distribution of the eigenvalues for $n=1000$.}
 \label{figUniformChain}
\end{figure}
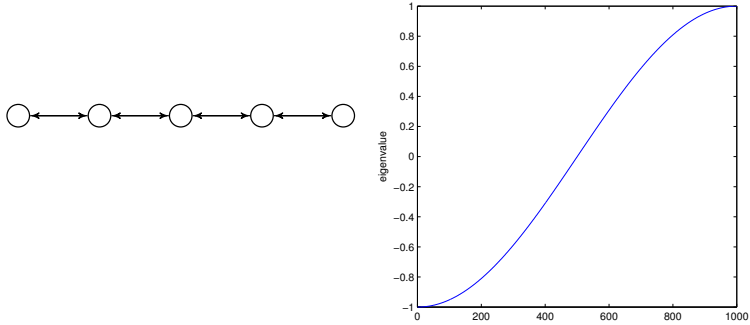

\captionof{table}{Uniform chain, second eigenvector.} 
\resizebox{0.9\textwidth}{!}{
\begin{minipage}{\linewidth}
\begin{center} 

 \label{tableUniformChain2}
\small\begin{tabular}{|r | c c c c c c | c c c c|}
 \hline
  &  \multicolumn{6}{c|}{DS\&SM} & \multicolumn{4}{c|}{DAM} \\
 \hline
 n  & $\gamma_\text{eff}$ & $\gamma$ & it & $C_\text{op}$ & lev & $d$ & $\gamma$ & it & $C_\text{op}$ & lev \\ 
 \hline\hline
 1024 & 0.06 & 1.5e-3 & 4 & 1.99 & 7 & 0.5 & 0.96 & $>50$ & 1.99 & 7 \\
 \hline
 4096 & 0.02 & 4.3e-4 & 3 & 2.00 & 9 & 0.5 & 0.99 & $>50$ & 1.99 & 8 \\
 \hline
 16384 & 0.02 & 6.1e-3 & 3 & 2.00 & 11 & 0.5 & 0.99 & $>50$ & 2.00 & 10 \\
 \hline
 32768 & 0.10 & 0.09 & 5 & 2.00 & 11 & 0.5 & 0.99 & $>50$ & 2.00 & 11\\
 \hline
 65536 & 0.06 & 0.06 & 4 & 2.00 & 12 & 0.5 & 0.99 & $>50$ & 2.00 & 12\\
 \hline
 262144 & 0.06 & 0.03 & 4 & 2.00 & 14 & 0.5 & 0.99 & $>50$ & 2.00 & 14\\
 \hline
\end{tabular}
\end{center}
\end{minipage}}

\subsection{Uniform Chain with Weak Link}
This test problem is a 1D chain with one weak link in the middle of the chain with weight $\epsilon = 0.001$ and uniform weights $\omega = 1$ on the other edges. Figure \ref{figUniformChainWeak} shows a graph representation and the distribution of eigenvalues. We can use the same parameters as for the uniform chain.

As for the standard uniform chain without a weak link, DS\&SM performs well and has very good convergence properties. It by far outperforms DAM.

\begin{figure}[]
\scalebox{0.6}{
\begin{tikzpicture}[->,>=stealth',shorten >=1pt,auto,node distance=1.8cm,
                    thick,main node/.style={circle,draw,font=\sffamily\Large\bfseries}]
  \coordinate (cone) at (0,0);
  \node[main node, minimum size = 0cm, draw=none] (0) [] {}; %fix place origin
  \node[main node, minimum size = 0.5cm] (1) at (0,5) {};
  \node[main node, minimum size = 0.5cm] (2) [right of=1] {};
  \node[main node, minimum size = 0.5cm] (3) [right of=2] {};
  \node[main node, minimum size = 0.5cm] (4) [right of=3] {};
  \node[main node, minimum size = 0.5cm] (5) [right of=4] {};
  \node[main node, minimum size = 0.5cm] (6) [right of=5] {};

  \path[every node/.style={font=\sffamily\small}]
    (1) edge [left] node[above] {} (2)
    (2) edge [left] node[above] {} (3)
        edge [left] node[below] {} (1)
    (3) edge [left] node[above] { \scalebox{1.43}{$\epsilon$} } (4)
        edge [left] node[below] {} (2)
    (4) edge [left] node[above] {} (5)
        edge [left] node[below] {} (3)
    (5) edge [left] node[above] {} (6)
        edge [left] node[below] {} (4)
   	(6) edge [left] node[below] {} (5);
\end{tikzpicture}}
\includegraphics[scale=0.4]{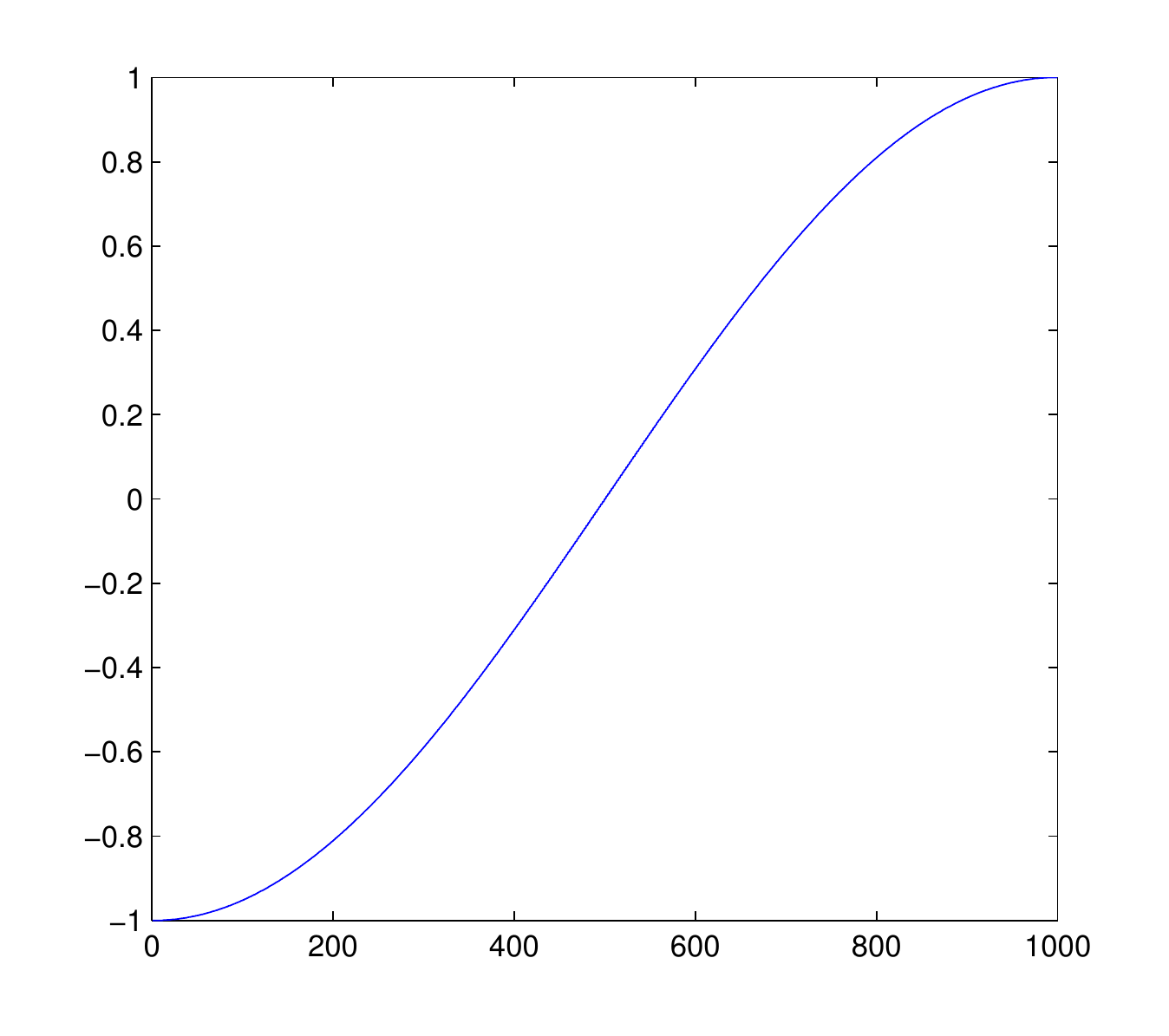}
\caption{Left: Graph representation of the uniform chain with one weak link. Right: Distribution of the eigenvalues for $n=1000$.}
 \label{figUniformChainWeak}
\end{figure}
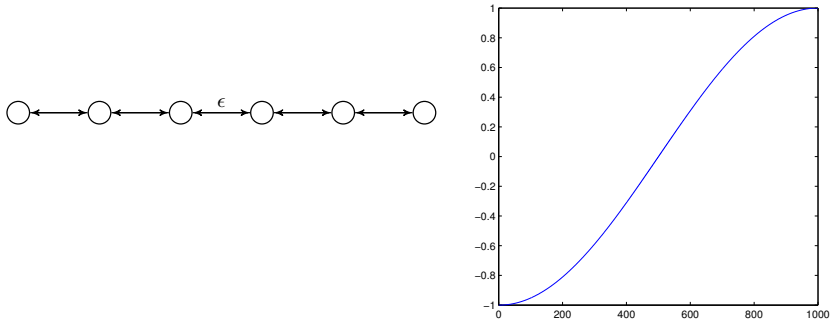

\captionof{table}{Uniform chain with one weak link, second eigenvector.} 
\resizebox{0.9\textwidth}{!}{
\begin{minipage}{\linewidth}
\begin{center} 
 \label{tableUniformChainWeak2}
\small\begin{tabular}{|r | c c c c c c | c c c c|}
 \hline
  & \multicolumn{6}{c|}{DS\&SM} & \multicolumn{4}{c|}{DAM} \\
 \hline
 n & $\gamma_\text{eff}$ & $\gamma$ & it & $C_\text{op}$ & lev & $d$ & $\gamma$ & it & $C_\text{op}$ & lev \\ 
 \hline\hline
 1024 & 0.02 & 2.6e-3 & 3 & 1.98 & 6 & 0.5 & 0.93 & $>50$ & 1.98 & 6 \\
 \hline
 4096 & 0.02 & 0.01 & 3 & 1.99 & 8 & 0.5 & 0.99 & $>50$ & 1.99 & 8 \\
 \hline
 16384 & 0.02 & 0.02 & 3 & 2.00 & 10 & 0.5 & 0.99 & $>50$ & 1.99 & 10 \\
 \hline
 32768 & 0.02 & 0.06 & 3 & 2.00 & 11 & 0.5 & 0.99 & $>50$ & 1.99 & 11 \\
 \hline
 65536 & 0.10 & 3.8e-3 & 5 & 2.00 & 12 & 0.5 & 0.99 & $>50$ & 1.99 & 12 \\
 \hline
 262144 & 0.02 & 1.3e-3 & 3 & 2.00 & 14 & 0.5 & 0.99 & $>50$ & 1.99 & 14 \\
 \hline
\end{tabular}
\end{center}
\end{minipage}}

\subsection{2D Lattice}
The 2D-lattice is another simple structured problem. In figure \ref{fig2Dlattice} we see the graph representation and distribution of eigenvalues. For the numerical tests we consider quadratic lattices.
(Only for $n=32768$ we have a 128x256 lattice.)\\
We use aggregates of size $s=4$ and chose the other parameters as for the uniform chain. 
Regarding the second eigenvector we observe that the method DS\&SM has reasonable convergence rates that increase as $n$ grows. $C_\text{op}$ seems to be bounded. The square and stretch approach significantly accelerates the convergence compared to DAM.

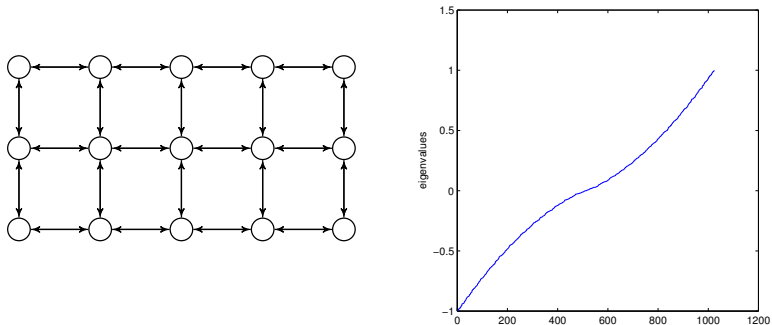
\begin{figure}[]
\scalebox{0.6}{
\begin{tikzpicture}[->,>=stealth',shorten >=1pt,auto,node distance=1.8cm,
                    thick,main node/.style={circle,draw,font=\sffamily\Large\bfseries}]
  \coordinate (cone) at (0,0);
  \node[main node, minimum size = 0cm, draw=none] (0) [] {}; %fix place origin
  \node[main node, minimum size = 0.5cm] (1) at (1,5.5) {};
  \node[main node, minimum size = 0.5cm] (2) [right of=1] {};
  \node[main node, minimum size = 0.5cm] (3) [right of=2] {};
  \node[main node, minimum size = 0.5cm] (4) [right of=3] {};
  \node[main node, minimum size = 0.5cm] (5) [right of=4] {};
  \node[main node, minimum size = 0.5cm] (6) [below of=1] {};
  \node[main node, minimum size = 0.5cm] (7) [below of=2] {};
  \node[main node, minimum size = 0.5cm] (8) [below of=3] {};
  \node[main node, minimum size = 0.5cm] (9) [below of=4] {};
  \node[main node, minimum size = 0.5cm] (10) [below of=5] {};
  \node[main node, minimum size = 0.5cm] (11) [below of=6] {};
  \node[main node, minimum size = 0.5cm] (12) [below of=7] {};
  \node[main node, minimum size = 0.5cm] (13) [below of=8] {};
  \node[main node, minimum size = 0.5cm] (14) [below of=9] {};
  \node[main node, minimum size = 0.5cm] (15) [below of=10] {};

  \path[every node/.style={font=\sffamily\small}]
    (1) edge [left] node[above] {} (2)
        edge [left] node[right] {} (6)
    (2) edge [left] node[above] {} (3)
        edge [left] node[below] {} (1)
        edge [left] node[right] {} (7)
    (3) edge [left] node[above] {} (4)
        edge [left] node[below] {} (2)
        edge [left] node[right] {} (8)
    (4) edge [left] node[above] {} (5)
        edge [left] node[below] {} (3)
        edge [left] node[right] {} (9)
   	(5) edge [left] node[below] {} (4)
        edge [left] node[right] {} (10)
    (6) edge [left] node[above] {} (7)
    	edge [left] node[right] {} (11)
    	edge [left] node[above] {} (1)
    (7)	edge [left] node[above] {} (8)
    	edge [left] node[above] {} (12)
    	edge [left] node[above] {} (6)
    	edge [left] node[above] {} (2)
    (8)	edge [left] node[above] {} (9)
    	edge [left] node[above] {} (13)
    	edge [left] node[above] {} (7)
    	edge [left] node[above] {} (3)
    (9)	edge [left] node[above] {} (10)
    	edge [left] node[above] {} (14)
    	edge [left] node[above] {} (8)
    	edge [left] node[above] {} (4)
    (10)edge [left] node[above] {} (15)
    	edge [left] node[above] {} (9)
    	edge [left] node[above] {} (5)
    (11)edge [left] node[above] {} (12)
    	edge [left] node[right] {} (6)
    (12)edge [left] node[above] {} (13)
		edge [left] node[above] {} (11)
		edge [left] node[right] {} (7)
    (13)edge [left] node[above] {} (14)
		edge [left] node[above] {} (12)
		edge [left] node[right] {} (8)
    (14)edge [left] node[above] {} (15)
    	edge [left] node[above] {} (13)
    	edge [left] node[right] {} (9)
    (15)edge [left] node[above] {} (14)
    	edge [left] node[right] {} (10);
\end{tikzpicture}}
\includegraphics[scale=0.4, trim = -1cm +1cm 0 0]{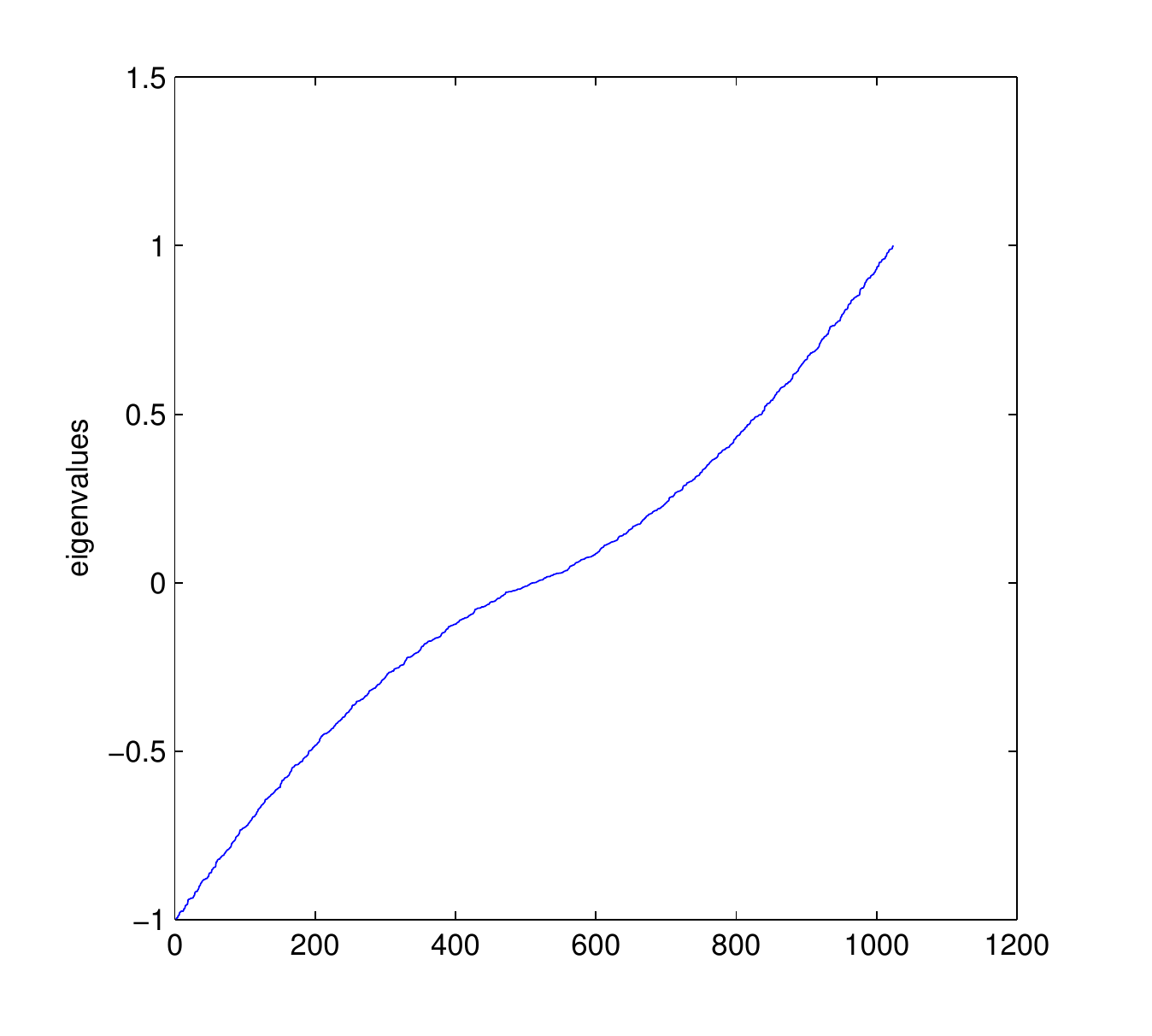}
\caption{Left: Graph of uniform 3x5 lattice. Right: Distribution of the eigenvalues for a 32x32 lattice.}
\label{fig2Dlattice}
\end{figure}

\captionof{table}{2D-lattice, second eigenvector.} 
\resizebox{0.9\textwidth}{!}{
\begin{minipage}{\linewidth}
\begin{center} 
 \label{table2D-lattice2}
\small\begin{tabular}{|r | c c c c c c | c c c c|}
 \hline
  & \multicolumn{6}{c|}{DS\&SM} & \multicolumn{4}{c|}{DAM} \\
 \hline
 n & $\gamma_\text{eff}$ & $\gamma$ & it & $C_\text{op}$ & lev & $d$ & $\gamma$ & it & $C_\text{op}$ & lev \\ 
 \hline\hline
 1024 & 0.14 & 0.09 & 7 & 1.67 & 4 & 0.5 & 0.15 & 8 & 1.39 & 4 \\
 \hline
 4096 & 0.37 & 0.33 & 13 & 1.76 & 5 & 0.5 & 0.48 & 15 & 1.37 & 5 \\
 \hline
 16384 & 0.36 & 0.40 & 14 & 1.63 & 6 & 0.5 & 0.77 & 45 & 1.37 & 6 \\
 \hline
 32768 & 0.45 & 0.56 & 18 & 1.59 & 6 & 0.5 & 0.83 & $>50$ & 1.33 & 6 \\
 \hline
 65536 & 0.50 & 0.55 & 20 & 1.64 & 7 & 0.5 & 0.90 & $>50$ & 1.36 & 7 \\
 \hline
\end{tabular}
\end{center}
\end{minipage}}

\subsection{Random Delaunay Triangulation}
The problems we considered so far had a simple geometric structure which implicitly prescribed the algebraically generated aggregates. The (undirected) graph for this example problem is a Delaunay triangulation for $n$ randomly chosen points in the $[0,1]^2$ unit square with uniform weights. Figure \ref{figRandomWalk} shows an example of such a graph and the spectrum of the corresponding Markov chain.\\

For $n=262144$ we used the aggregation parameter $\theta = 0.1$ instead of $0.25$. We use aggregates of size $4$ and besides that the standard parameters. For D\&SM we use a higher number of 300 relaxation steps.\\

Regarding the second eigenvector we observe that the convergence factors approach $1$ and the method becomes less effective for large $n$. Moreover DS\&SM does not significantly outperform standard DAM. But still the method is superior to plain relaxation.
The inferior results can be explained by the observation that for this problem the coarse level correction introduces new errors which need to be reduced by additional relaxation steps.
The random grid makes the choice of good aggregates more difficult.
If the aggregation algorithm is not optimal, this leads to slower convergence.
Note here, that the convergence is not caused by the large number of relaxation steps, but by relaxation combined with coarse level correction. Plain relaxation alone does not significantly change the residual at all for large matrices.

\begin{figure}[H]
\includegraphics[scale=0.4, trim = +2cm 0 0 0]{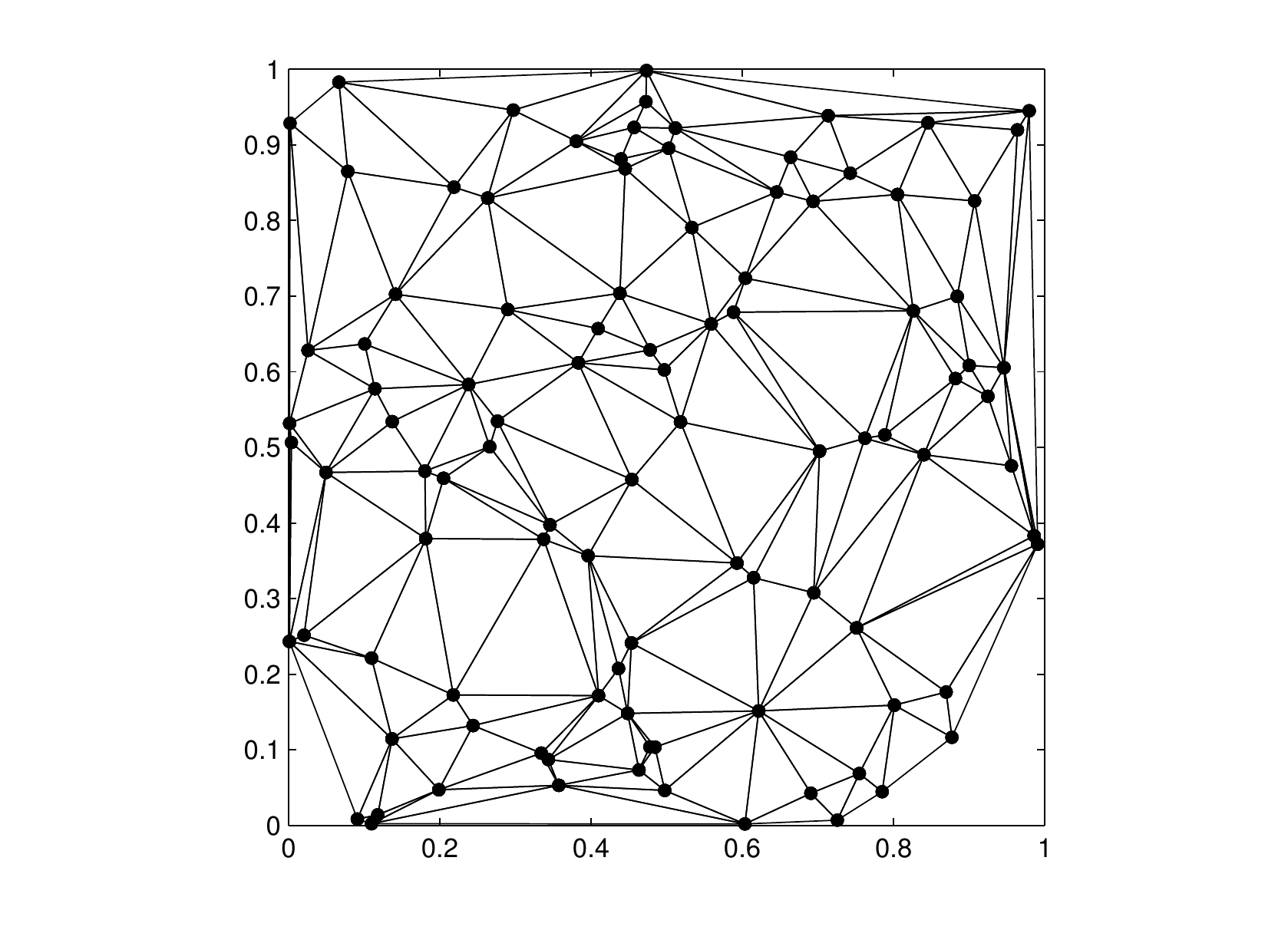}
\includegraphics[scale=0.4, trim = +3cm 0 0 0]{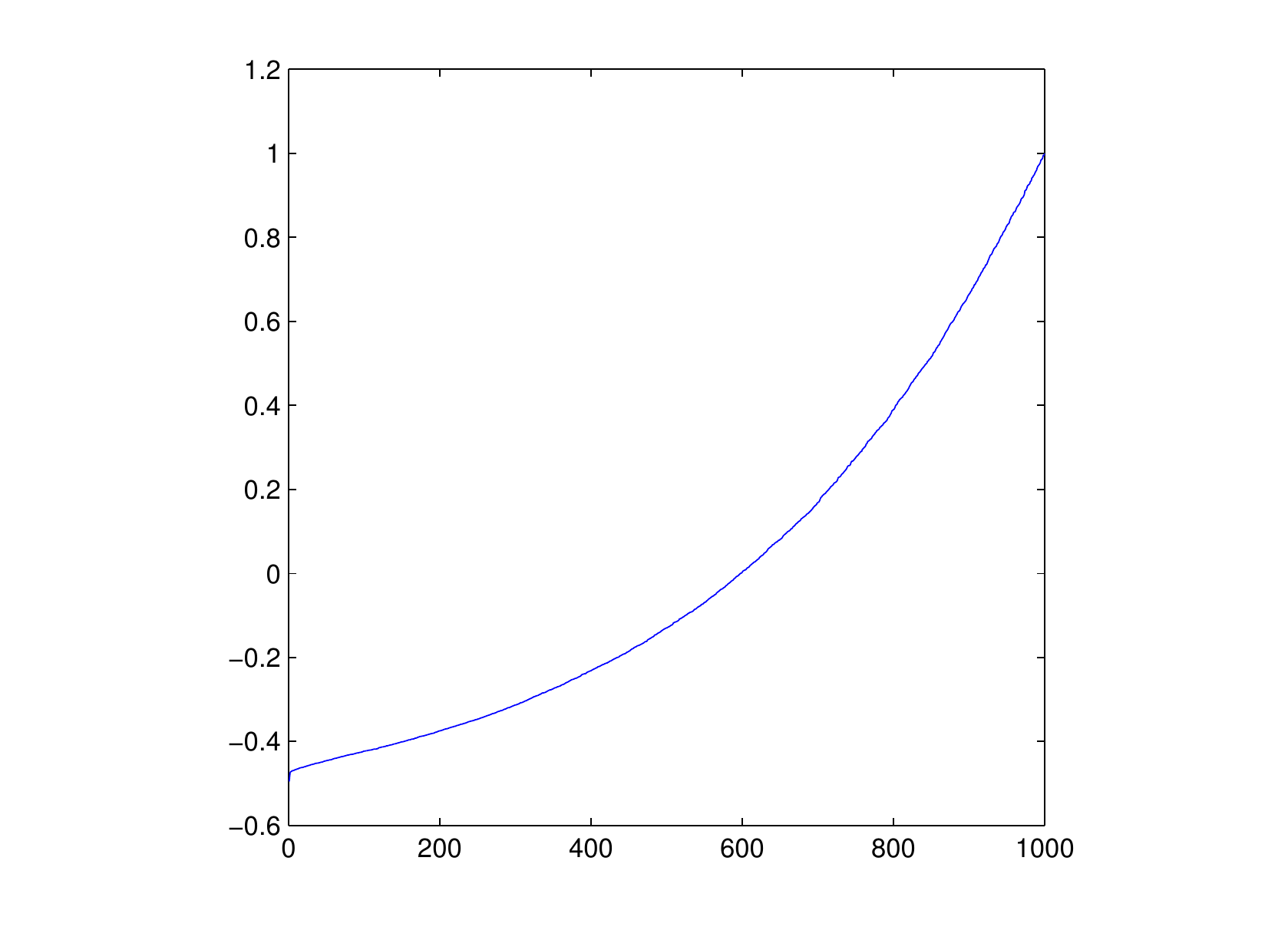}
\caption{Left: Example of a Delaunay triangulation for $n=100$. Right: Distribution of the eigenvalues for $n=1000$.}
\label{figRandomWalk}
\end{figure}

\captionof{table}{Random Walk, second eigenvector} 
\resizebox{0.9\textwidth}{!}{
\begin{minipage}{\linewidth}
\begin{center} 
 \label{tableRandomWalk2}
\small\begin{tabular}{|r | c c c c c c | c c c c|}
 \hline
  & \multicolumn{6}{c|}{DS\&SM} & \multicolumn{4}{c|}{DAM} \\
 \hline
 n & $\gamma_\text{eff}$ & $\gamma$ & it & $C_\text{op}$ & lev & $d$ & $\gamma$ & it & $C_\text{op}$ & lev \\ 
 \hline\hline
 1024 & 0.48 & 0.47 & 18 & 1.76 & 4 & 0.5 & 0.56 & 28 & 1.38 & 4 \\
 \hline
 4096 & 0.67 & 0.71 & 31 & 1.84 & 5 & 0.5 & 0.92 & $>50$ & 1.38 & 5 \\
 \hline
 16384 & 0.66 & 0.75 & 30 & 1.87 & 6 & 0.5 & 0.98 & $>50$ & 1.38 & 6 \\
 \hline
 32768 & - 	& 0.91 & $>50$ & 1.89 & 6 & 0.5 & 0.98 & $>50$ & 1.38 & 7 \\
 \hline
\end{tabular}
\end{center}
\end{minipage}}

\subsection{1D Multi-Well Potential}
The double-well potential is a simplified academic example problem. For instance, it is used to illustrate the reaction pathway of a protein folding process.
It is based on an energy potential with two wells (figure \ref{figureEnergyPotential}) on which a particle is moving around driven by both diffusion and drift induced by the energy potential. We discretize by dividing the domain $[0,1]$ into equidistant intervals and use a Monte Carlo method to estimate the transition probabilities between the boxes.

Note that for this test problem the number of non-zero entries in the transition matrix grows quadratically with $n$, whereas for all other test problems the number of non-zero entries is proportional to the matrix size. for $n=4096$ the matrix has about 2.3e+6 non-zero entries and for $n=16384$ the matrix has more than 3.7e+7 non-zero entries.\\

We use aggregates of size $2$ and take the average of the diagonal elements as the stretching parameter.\\
Note that there is a spectral gap after the second eigenvalue. Hence the second eigenvector is easier to compute. We use only 3x pre- and post-relaxation steps to emphasize the convergence effect of the coarse level improvement. Also square and stretch does not significantly improve the convergence.\\

The method can also be successfully applied to a four-wells potential constructed similarly as the double-well potential. The four-well potential however has more than one eigenvalue close to $1$, hence the computation of the second eigenvalue is harder. For the multilevel V-cycle we use aggregates of size $3$ and $9$ pre- and post-relaxation steps. In table \ref{fourWellPotentialTable} we can observe that the convergence factor is independent of $n$ and $C_{op}$ grows slightly with $n$.

\begin{figure}[]
\includegraphics[scale=0.40, trim = 0 0 0 0]{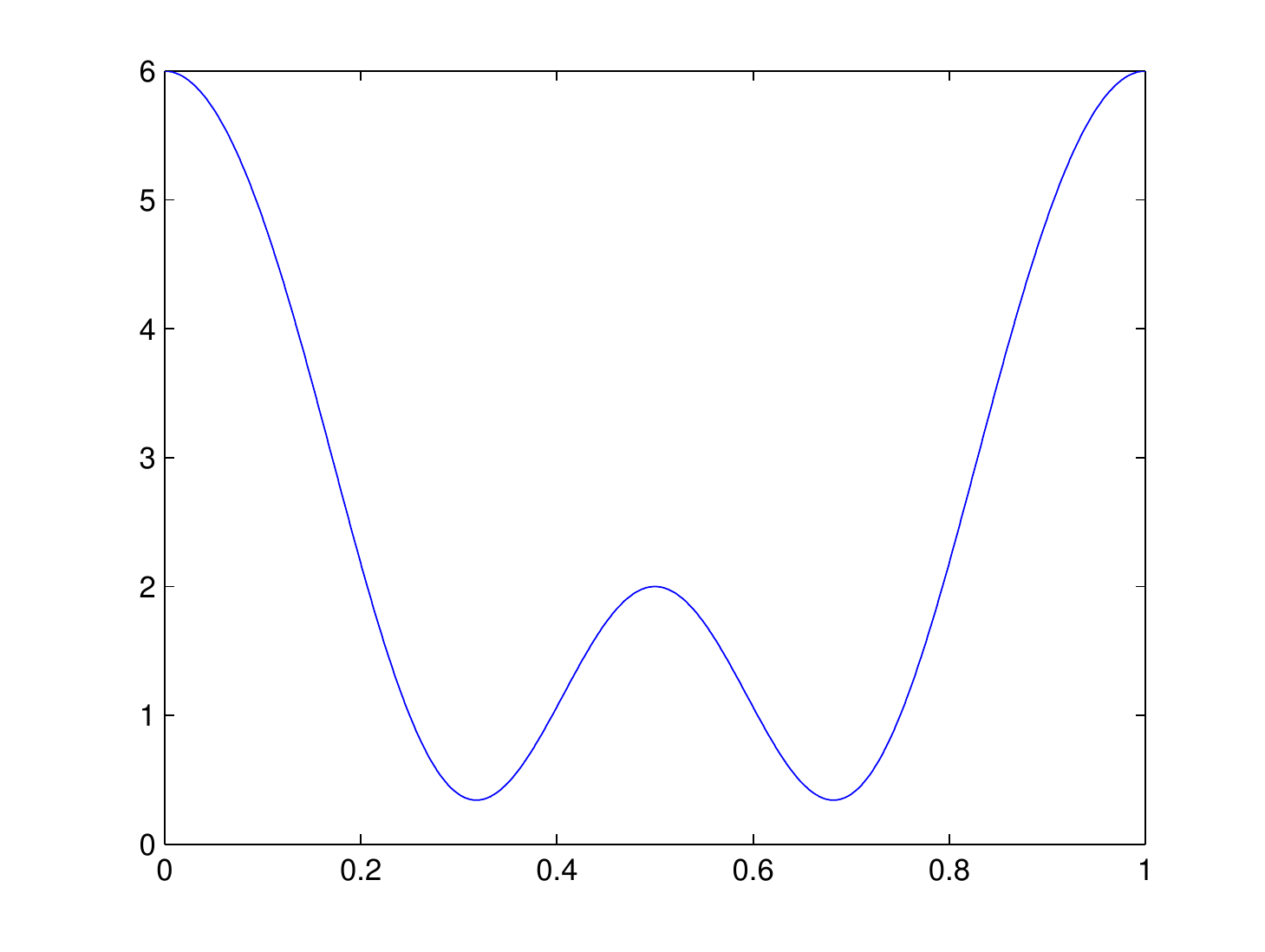}
\includegraphics[scale=0.40, trim = 0 0 0 0]{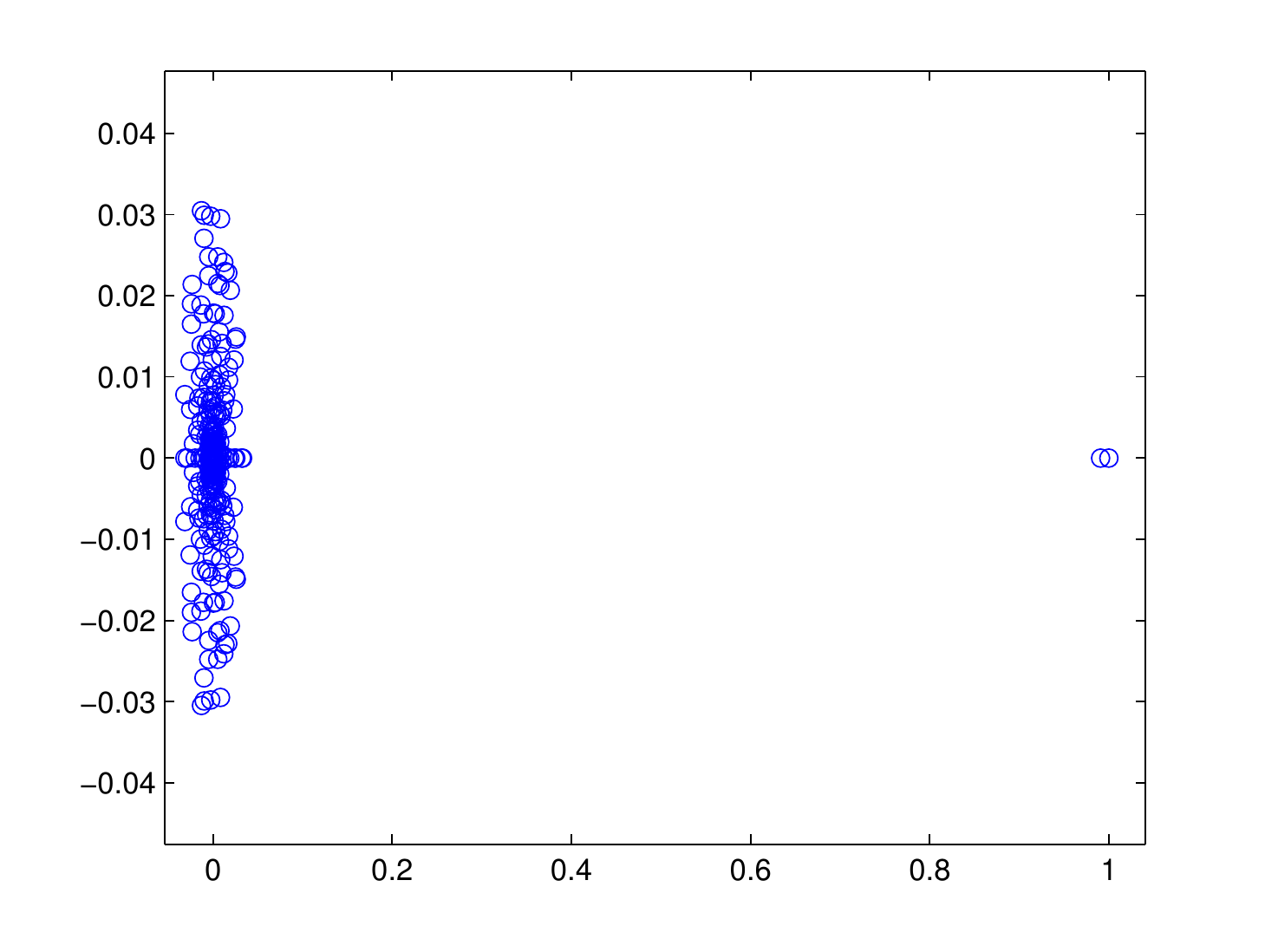}
\caption{Left: Double-well potential. Right: Eigenspectrum of the discretized transition matrix. There is a spectral gap after the second eigenvector. This is due to the existence of the two metastable sets induced by the wells.}
\end{figure}

\begin{figure}[]
\includegraphics[scale=0.40, trim = 0 +1cm 0 0]{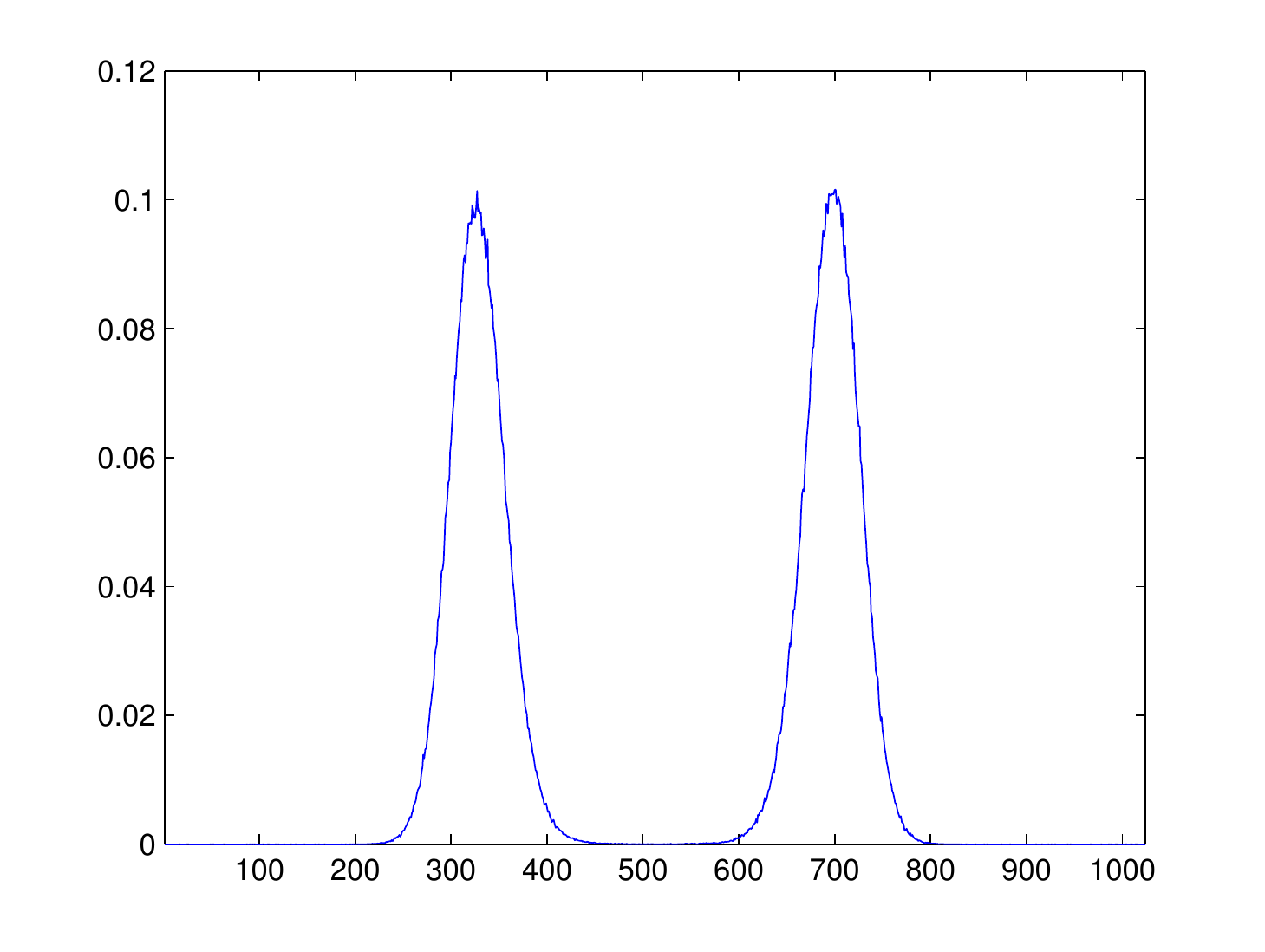}
\includegraphics[scale=0.40, trim = 0 +1cm 0 0]{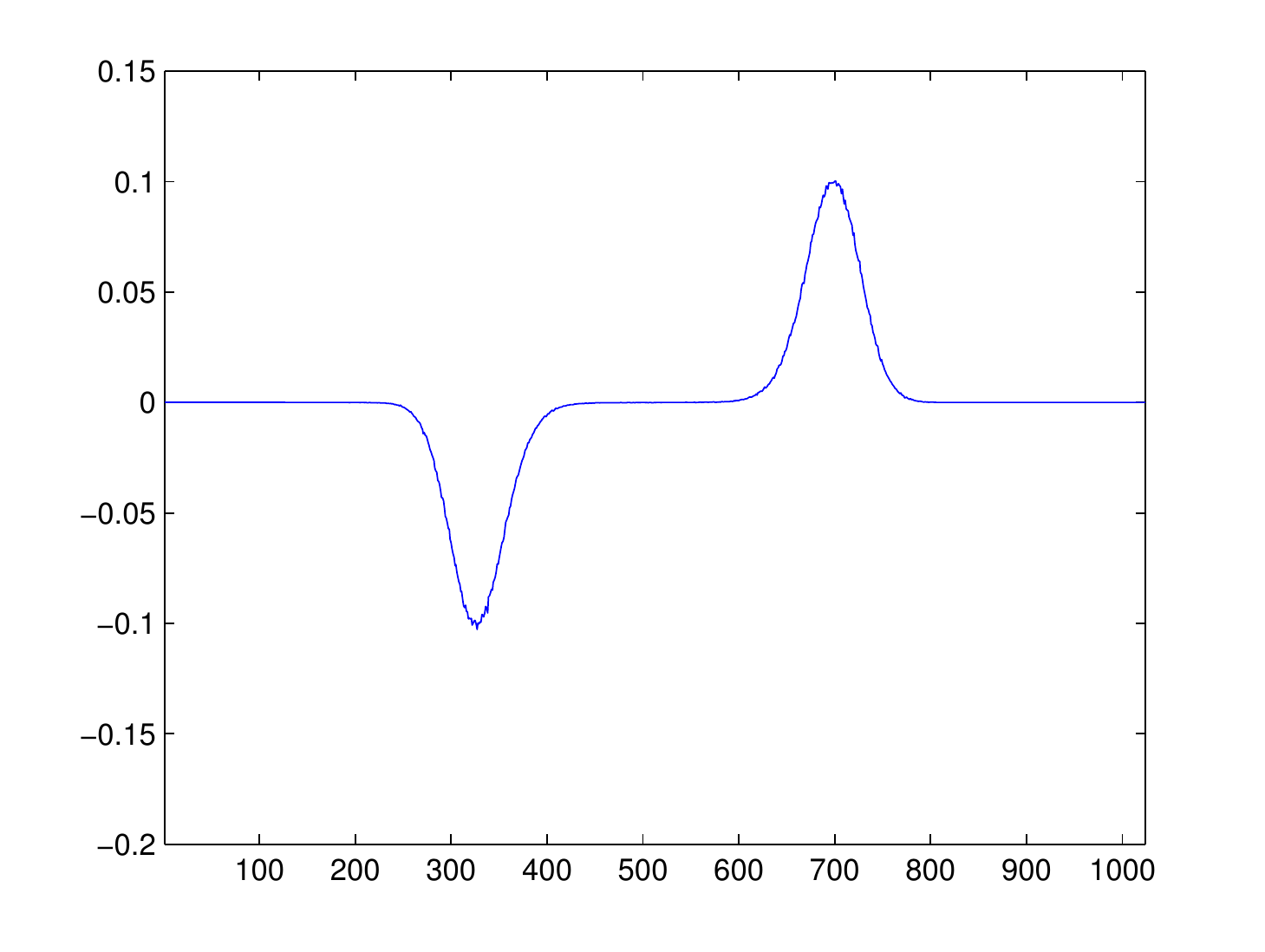}
\caption{Left: Stationary distribution of the double-well potential. Right: Second eigenvector corresponding to the characteristic slowest dynamic of the process.}
\label{figureEnergyPotential}
\end{figure}

\captionof{table}{Double-well potential, second eigenvector.}
\resizebox{0.9\textwidth}{!}{
\begin{minipage}{\linewidth}
\begin{center} 
 \label{tableTwoWell2}
\small\begin{tabular}{|r | c c c c c c | c c c c|}
 \hline
  &  \multicolumn{6}{c|}{DS\&SM} & \multicolumn{4}{c|}{DAM} \\
 \hline
 n & $\gamma_\text{eff}$ & $\gamma$ & it & $C_\text{op}$ & lev & $d$ & $\gamma$ & it & $C_\text{op}$ & lev \\ 
 \hline\hline
 512 & 0.10 & 0.03 & 7 & 1.45 & 4 & 0.45 & 0.04 & 6 & 1.23 & 4 \\
 \hline
 1024 & 0.15 & 0.10 & 8 & 1.54 & 4 & 0.45 & 0.11 & 7 & 1.29 & 5 \\
 \hline
 2048 & 0.17 & 0.10 & 8 & 1.60 & 5 & 0.44 & 0.12 & 8 & 1.34 & 6 \\
 \hline
 4096 & 0.10 & 0.08 & 6 & 1.66 & 6 & 0.43 & 0.10 & 7 & 1.34 & 6 \\
 \hline
\end{tabular}
\end{center}
\end{minipage}}

\begin{figure}[H]
\includegraphics[scale=0.40, trim = 0 0 0 0]{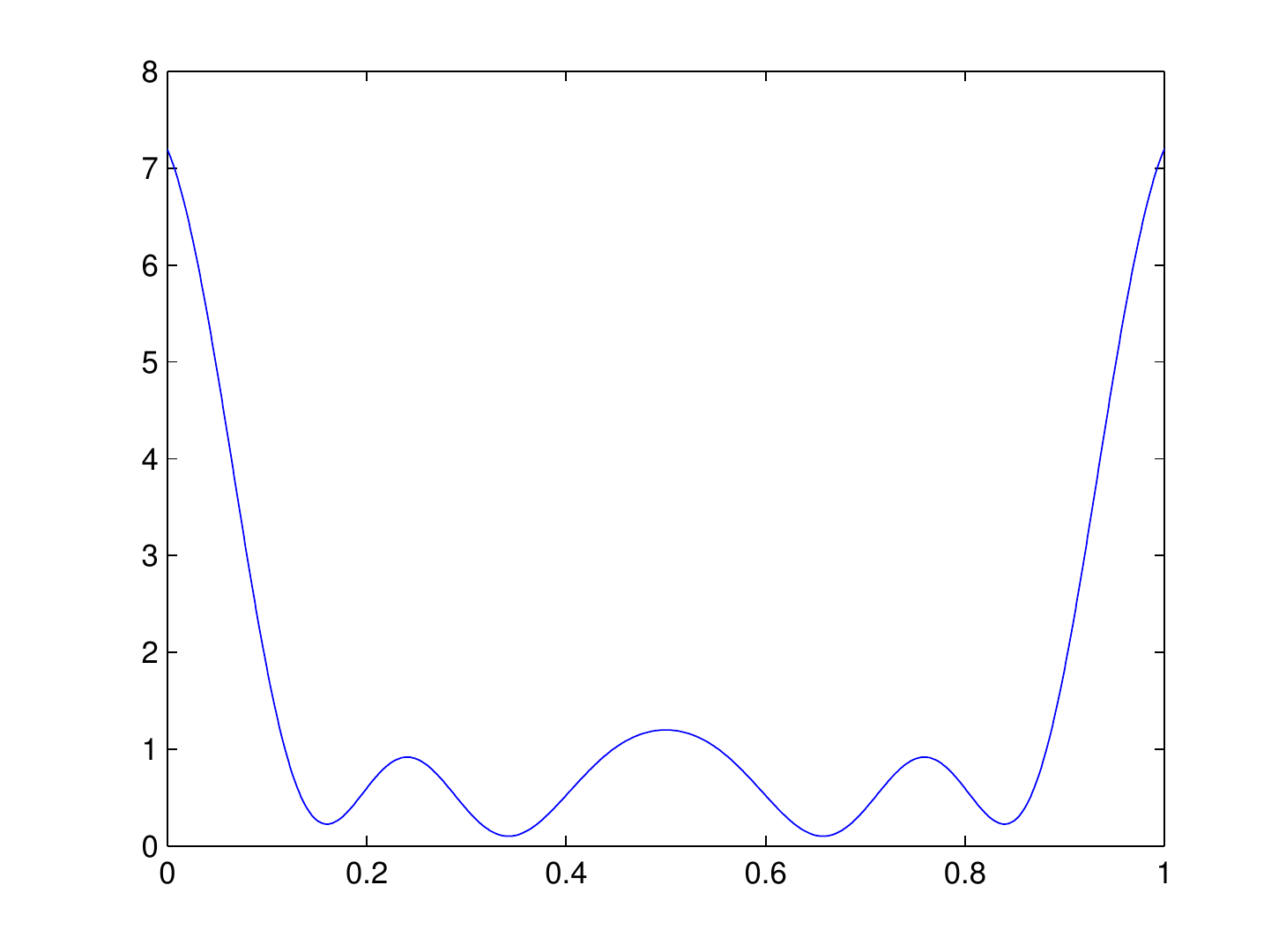}
\includegraphics[scale=0.40, trim = 0 0 0 0]{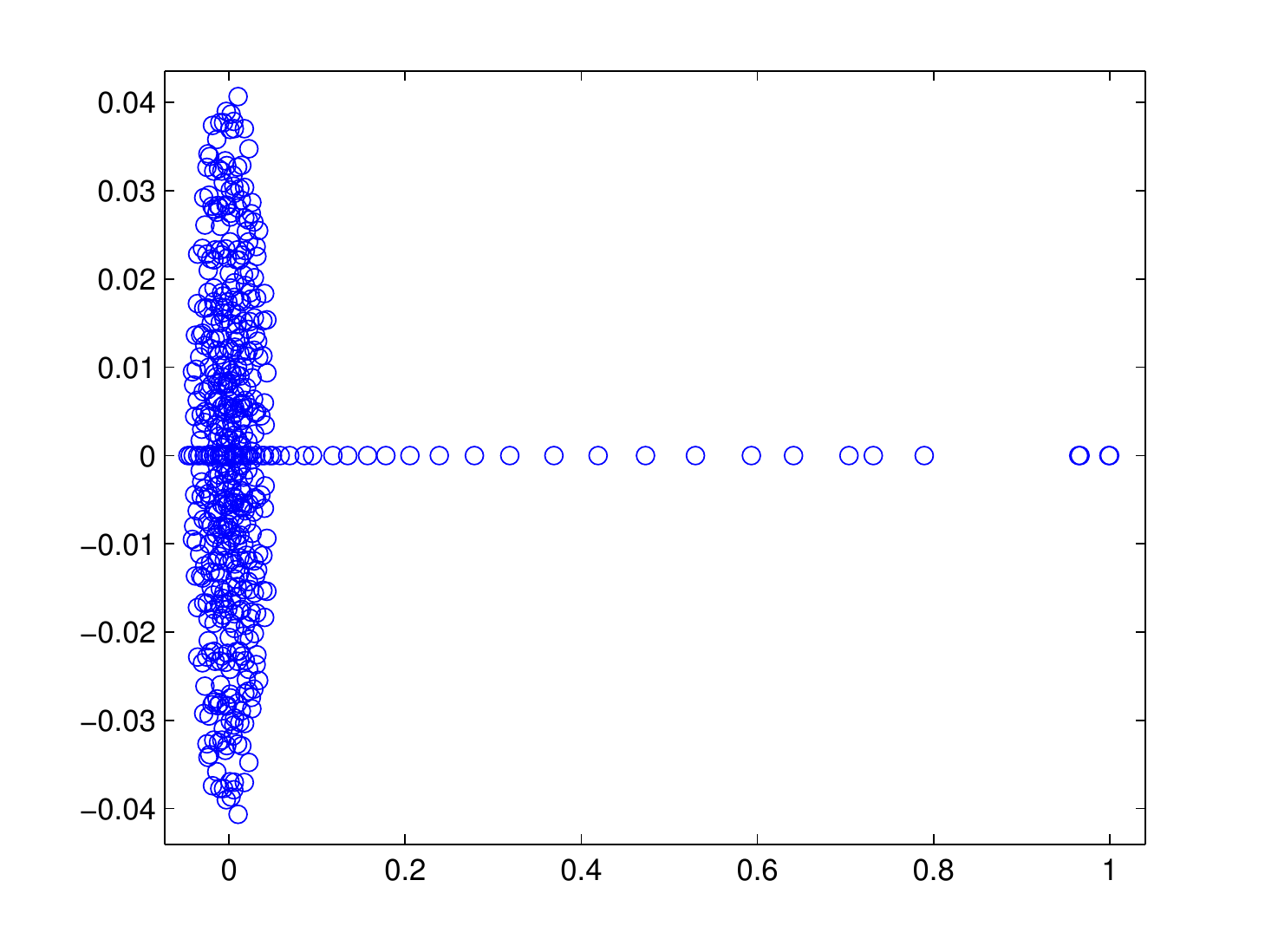}
\caption{Left: Four-well potential. Right: Eigenspectrum of the discretized transition matrix. The two dots close to 1 are in fact four (overlapping) dots.}
\end{figure}

\captionof{table}{Four-well potential, second eigenvector.} 
\label{fourWellPotentialTable}
\resizebox{0.9\textwidth}{!}{
\begin{minipage}{\linewidth}
\begin{center} 
\small\begin{tabular}{|r | c c c c c c | c c c c|}
 \hline
  &  \multicolumn{6}{c|}{DS\&SM} & \multicolumn{4}{c|}{DAM} \\
 \hline
 n & $\gamma_\text{eff}$ & $\gamma$ & it & $C_\text{op}$ & lev & $d$ & $\gamma$ & it & $C_\text{op}$ & lev \\ 
 \hline\hline
 512 & 0.17 & 0.23 & 9 & 1.46 & 4 & 0.45 & 0.36 & 15 & 1.24 & 4 \\
 \hline
 1024 & 0.11 & 0.11 & 7 & 1.49 & 4 & 0.45 & 0.27 & 13 & 1.26 & 4 \\
 \hline
 2048 & 0.15 & 0.17 & 8 & 1.52 & 5 & 0.44 & 0.31 & 14 & 1.27 & 5 \\
 \hline
 4096 & 0.19 & 0.19 & 9 & 1.55 & 6 & 0.43 & 0.32 & 15 & 1.27 & 6 \\
 \hline
\end{tabular}
\end{center}
\end{minipage}}

\subsection{Complex Uniform Chain}

\label{artificialChain}

This example is a modified uniform chain with additional transitions along the chain as depicted in figure \ref{figArtificial}.
The non-symmetry of the sparsity pattern of the transition matrix causes complex eigenvalues. We see that the eigenvalues with large imaginary part are contained in a ball of radius $\approx 0.5$ around zero. The eigenvalues close to $1$ remain real-valued.

We choose aggregates of size $3$. The stretching parameter $d$ has to be smaller than $0.5$ due to the imaginary spectrum.
Regarding the second eigenvector we observe that DS\&SM outperforms DAM and has a bounded operator complexity. However the convergence factor increases as $n$ grows.

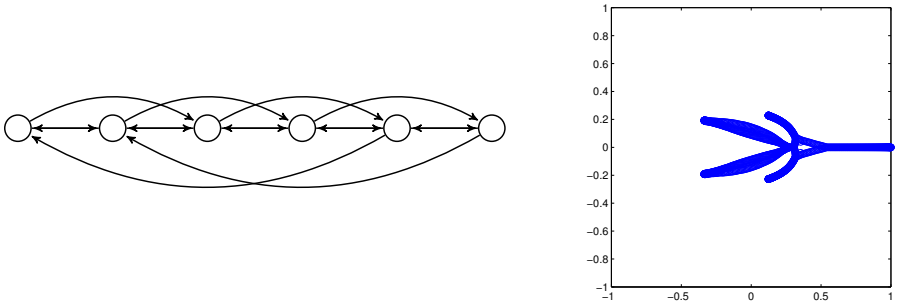
\begin{figure}[H]
\centering
\scalebox{0.7}{
\begin{tikzpicture}
	[->,>=stealth',shorten >=1pt,auto,node distance=1.8cm,
                    thick,main node/.style={circle,draw,font=\sffamily\Large\bfseries}]

  \coordinate (cone) at (0,0);
  \node[main node, minimum size = 0cm, draw=none] (0) [] {}; %fix place origin
  \node[main node, minimum size = 0.5cm] (1) at (0,3) {};
  \node[main node, minimum size = 0.5cm] (2) [right of=1] {};
  \node[main node, minimum size = 0.5cm] (3) [right of=2] {};
  \node[main node, minimum size = 0.5cm] (4) [right of=3] {};
  \node[main node, minimum size = 0.5cm] (5) [right of=4] {};
  \node[main node, minimum size = 0.5cm] (6) [right of=5] {};

  \path[every node/.style={font=\sffamily\small}]
    (1) edge [left] node[above] {} (2)
    	edge [bend left] node[] {} (3)
    (2) edge [left] node[above] {} (3)
        edge [left] node[below] {} (1)
        edge [bend left] node[] {} (4)
    (3) edge [left] node[above] {} (4)
        edge [left] node[below] {} (2)
        edge [bend left] node[] {} (5)
    (4) edge [left] node[above] {} (5)
        edge [left] node[below] {} (3)
        edge [bend left] node[] {} (6)
   	(5) edge [left] node[below] {} (4)
   		edge [left] node[below] {} (6)
   		edge [bend left] node[] {} (1)
   	(6) edge [left] node[below] {} (5)
   		edge [bend left] node[] {} (2);

\end{tikzpicture}}
\includegraphics[scale=0.41, trim = 0 +1cm 0 0]{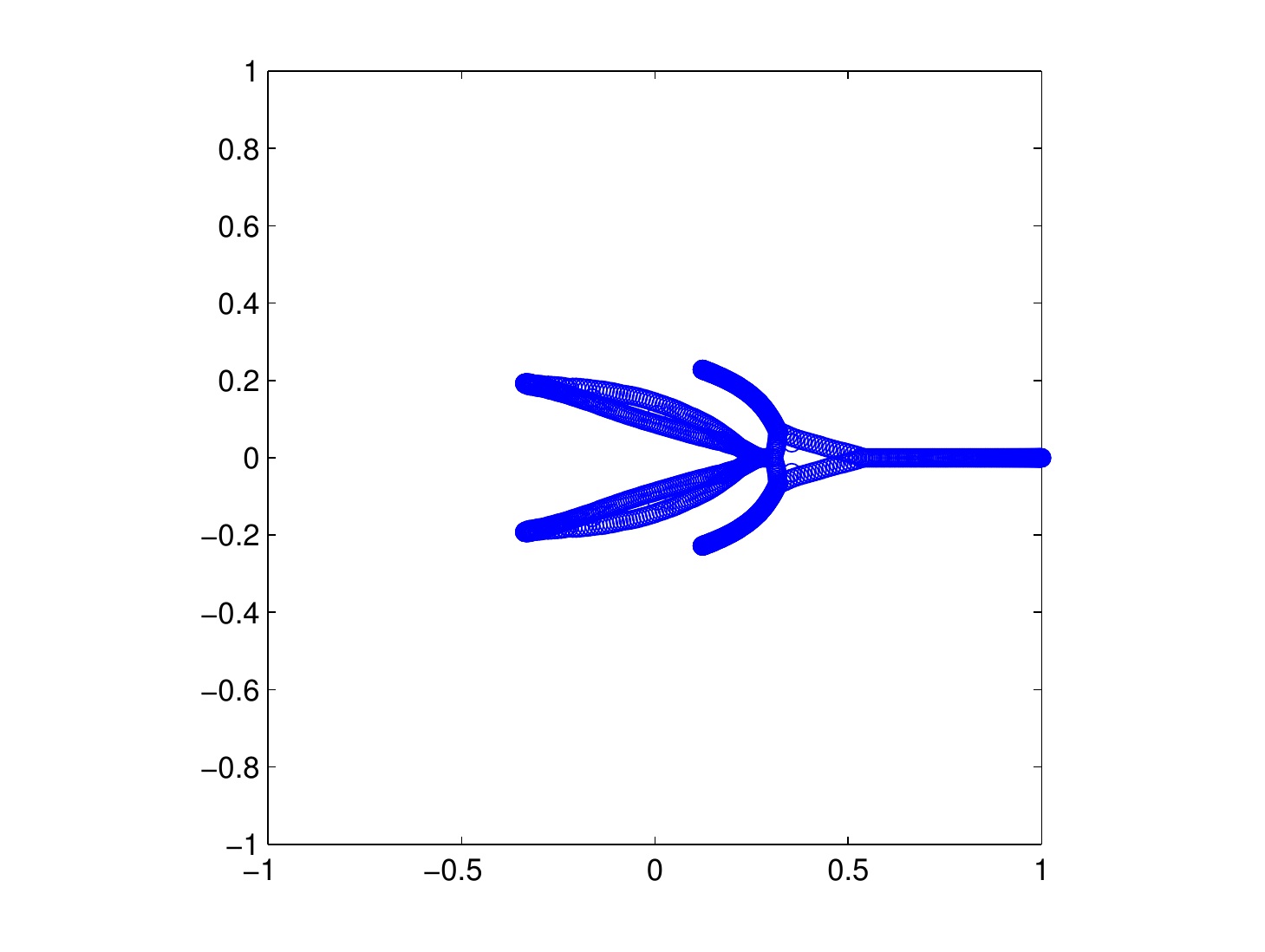}
\caption{Left: Directed graph of complex uniform chain. Right: Complex spectrum for $n=1000$.}
\label{figArtificial}
\end{figure}

\vspace{1em}
\captionof{table}{Complex uniform chain, second eigenvector.} \label{tableArtificialChain2}
\resizebox{0.9\textwidth}{!}{
\begin{minipage}{\linewidth}
\begin{center} 
\small\begin{tabular}{|r | c c c c c c | c c c c|}
 \hline
  &  \multicolumn{6}{c|}{DS\&SM} & \multicolumn{4}{c|}{DAM} \\
 \hline
 n & $\gamma_\text{eff}$ & $\gamma$ & it & $C_\text{op}$ & lev & $d$ & $\gamma$ & it & $C_\text{op}$ & lev \\ 
 \hline\hline
 1024 & 0.23 & 0.41 & 12 & 1.32 & 5 & 0.46 & 0.86 & $>50$ & 1.18 & 5 \\
 \hline
 4096 & 0.44 & 0.69 & 21 & 1.32 & 7 & 0.46 & 0.98 & $>50$ & 1.19 & 7 \\
 \hline
 16384 & 0.38 & 0.76 & 18 & 1.32 & 9 & 0.46 & 0.99 & $>50$ & 1.19 & 9 \\
 \hline
 32768 & 0.55 & 0.89 & 29 & 1.32 & 10 & 0.45 & 0.99 & $>50$ & 1.19 & 10 \\
 \hline
\end{tabular}
\end{center}
\end{minipage}
} 

\section{Conclusions}
The new method DS\&SM to compute the second eigenvector has superior convergence properties for matrices with few connections like the uniform chain, the uniform chain with a weak link and the complex uniform chain. Also for structured problems like the 2D-lattice the methods is applicable. For the multi-well potential it appears that the method is very scalable and has good convergence results. However the method is limited by its relatively poor convergence for the unstructured Delaunay triangulation with large $n$.
A reason for this behavior might be that the aggregation method is not optimal for this type of grid. But the method is still far more effective than ordinary relaxation and improves the convergence. Also it is possible that a different aggregation method could significantly improve the convergence.\\

Regarding the analysis, it would be interesting if the assumption on slow processes can be proven for certain classes of stochastic matrices and aggregation methods.\\
So far the theory for smooth errors has been developed largely for symmetric positive-definite matrices. I think it would be good if one could extend the theory to arbitrary relaxation methods with a thorough framework. 

Finally one could extend the method to compute the first $k$ dominant eigenvectors with a multilevel approach. At least for problems with few connections like the uniform chain, this approach appears to be very promising.

\bibliography{literatur}

\end{document}